\documentclass[12pt,reqno]{amsart}
\newcommand\theoremnumbering{subsection}   
\usepackage{fullpage,setspace}
\usepackage{mathrsfs} 
\newcommand\choosefont[1]{\usepackage{#1}}
\usepackage[bottom]{footmisc}  
\usepackage{enumerate}   
\usepackage{url,cite,fancyheadings}
\usepackage{enumerate}
\usepackage{asymptote}
\usepackage{amsmath,amssymb,amscd,mathrsfs,latexsym}
\usepackage{mathdots}

\usepackage{ifthen}     

\newcommand\pubpri[2]{%
\ifthenelse{\equal{\version}{public}}%
{{#1}}%
{\ifthenelse{\equal{\finalized}{no}}{\marginpar{\scshape\small Pubpri Alert}{#2}}{#2}}{}}
\newcommand\pubprinoalert[2]{%
\ifthenelse{\equal{\version}{public}}%
{{#1}}%
{#2}}

\newcommand\ignore[1]{}

\usepackage{color}
\providecommand\wantcolor{yes}   %
\definecolor{backgroundyellow}{cmyk}{.2,.1,.8,.2}
\definecolor{backgroundblue}{rgb}{0,0,1}
\definecolor{backgroundred}{rgb}{1,0,0}
\definecolor{backgroundmagenta}{cmyk}{0,1,0,0}
\ifthenelse{\equal{\wantcolor}{no}}{\renewcommand\color[1]{}}{}

\newcommand\mysection{\section}

\newcommand\mysubsection{\subsection}
\newcommand\mysubsectionstar{\subsection*}
\newcommand\mysubsubsection[1]{%
		\subsubsection{\sffamily\upshape\mdseries #1}}

\newcommand\mysss{\mysubsubsection}

\usepackage{chngcntr}
\counterwithin{figure}{section}

\usepackage{url,hyperref}  
\usepackage{hyperref}
\hypersetup{colorlinks=true,
linkcolor=blue,
filecolor=magenta,
urlcolor=cyan,}
\providecommand{\theoremnumbering}{document}

\ifthenelse{\equal{\theoremnumbering}{section}}{\newtheorem{annotation}{Annotation}[section]}%
{\ifthenelse{\equal{\theoremnumbering}{subsection}}{\newtheorem{annotation}{Annotation}[subsection]}{\newtheorem{annotation}{Annotation}}}
\newtheorem{theorem}[annotation]{
		Theorem}
\newtheorem{lemma}[annotation]{
		Lemma}
\newtheorem{definition}[annotation]{
		Definition}
\newtheorem{corollary}[annotation]{
		Corollary}
\newtheorem{proposition}[annotation]{
		Proposition}

\newtheorem{example}[annotation]{
		Example}
\newcommand\bexample{\begin{example}\begin{rm}}
\newcommand\eexample{\end{rm}\hfill$\Box$\end{example}}
\newtheorem{examplenobox}[annotation]{
		Example}
\newcommand\bexamplenobox{\begin{examplenobox}\begin{rm}}
\newcommand\eexamplenobox{\end{rm}\end{examplenobox}}
\newtheorem{exercise}[annotation]{
		Exercise}
\newcommand\bexercise{\begin{exercise}\begin{rm}}
\newcommand\eexercise{\end{rm}\end{exercise}}
\newtheorem{notation}[annotation]{
		Notation}
\newcommand\bnotation{\begin{notation}\begin{rm}}
\newcommand\enotation{\end{rm}\end{notation}}

\newtheorem{remark}[annotation]{
		Remark}
\newcommand\bremark{\begin{remark}
\begin{upshape}}
\newcommand\eremark{\end{upshape}
\end{remark}}
\newenvironment{remark*}{%
\par\noindent{\scshape 
  Remark: }\begin{rm}}{\hfill\end{rm}\newline} 
\newcommand\bremarkstar{\begin{remark*}}
\newcommand\eremarkstar{\end{remark*}}
\newcommand\bdefn{\begin{definition}
\begin{upshape}}
\newcommand\edefn{\end{upshape}
\end{definition}}
\newtheorem{caveat}[annotation]{
		Caveat}
\newcommand\bcaveat{\begin{caveat}
\begin{upshape}}
\newcommand\ecaveat{\end{upshape}
\end{caveat}}
\newenvironment{caveatstar}{
\par\noindent{\scshape\bfseries
  Caveat: }\begin{rm}}{\end{rm}\newline} 
\newcommand\bcaveatstar{\begin{caveatstar}}
\newcommand\ecaveatstar{\end{caveatstar}}
\newenvironment{myproof}{%
\par\noindent{\scshape 
  Proof: }\begin{rm}}{\hfill$\Box$\end{rm}\newline} 
\newcommand\bmyproof{\begin{myproof}}
\newcommand\emyproof{\end{myproof}}
\newenvironment{myproofnobox}{%
\par\noindent{\scshape Proof: }\begin{rm}}{\end{rm}\hfill\newline}
\newcommand\bmyproofnobox{\begin{myproofnobox}}
\newcommand\emyproofnobox{\end{myproofnobox}}
\newenvironment{solution}{%
\par\noindent{\scshape Solution: }\begin{rm}}{\hfill$\Box$\end{rm}\newline}
\newenvironment{solutionnobox}{%
\par\noindent{\scshape Solution: }\begin{rm}}{\end{rm}}
\newcommand\bsolution{\begin{solution}\begin{rm}}
\newcommand\esolution{\end{rm}\end{solution}}
\newcommand\bsolutionnobox{\begin{solutionnobox}\begin{rm}}
\newcommand\esolutionnobox{\end{rm}\end{solutionnobox}}
\newcommand\bthm{\begin{theorem}}
\newcommand\ethm{\end{theorem}}
\newcommand\bcor{\begin{corollary}}
\newcommand\ecor{\end{corollary}}
\newcommand\blemma{\begin{lemma}}
\newcommand\elemma{\end{lemma}}
\newcommand\bprop{\begin{proposition}}
\newcommand\eprop{\end{proposition}}
\newcommand\beqn{\begin{equation}}
\newcommand\eeqn{\end{equation}}
\newcommand\beqnstar{\begin{equation*}}
\newcommand\eeqnstar{\end{equation*}}
\usepackage{amsthm}

\newcommand\mtitle[1]%
{\marginpar{\sffamily\raggedright\upshape\small #1}}

\providecommand\finalized{yes}
\ifthenelse{\equal{\finalized}{yes}}{}{\marginparwidth=2cm}
\ifthenelse{\equal{\finalized}{yes}}%
{\newcommand\mylabel[1]{\label{#1}}}%
{\newcommand\mylabel[1]{\label{#1}\marginpar{[{\ttfamily\upshape\tiny #1}]}}}
\ifthenelse{\equal{\finalized}{yes}}%
{\newcommand\checked[1]{}}%
{\newcommand\checked[1]{\marginpar{[{\ttfamily\upshape\tiny CHECKED: #1}]}}}
\ifthenelse{\equal{\finalized}{yes}}%
{\newcommand\spellchecked[1]{}}%
{\newcommand\spellchecked[1]{\marginpar{[{\ttfamily\upshape\tiny SPELLCHECKED: #1}]}}}

\providecommand\version{public}   
\ifthenelse{\equal{\version}{public}}%
{\newcommand\mcomment[1]{}}%
{\newcommand\mcomment[1]{\marginpar{{\raggedright\sffamily\upshape\small
\begin{spacing}{0.75} #1\end{spacing}}}}}
\ifthenelse{\equal{\version}{public}}%
{\newcommand\fcomment[1]{}}%
{\newcommand\fcomment[1]{\footnote{#1}}}
\ifthenelse{\equal{\version}{public}}%
{\newcommand\comment[1]{}}%
{\newcommand\comment[1]{{\small #1}}}

\choosefont{rsfso}    

\usepackage{amsthm}
\usepackage{empheq}
\usepackage{accents}
\usepackage[all, poly]{xy}
\usepackage[utf8]{inputenc}
\usepackage{mathtools}
\usepackage{soul}
\usepackage{tikz}
\usetikzlibrary{matrix,arrows,decorations.pathmorphing}
\usepackage[left=2.5cm,right=2.5cm,top=2cm,bottom=2cm,includeheadfoot]{geometry}

\theoremstyle{plain}
\newtheorem{prop}{Proposition}[section]
\newtheorem{lem}{Lemma}[section]
\newtheorem{rem}{Remark}[section]

\theoremstyle{remark}

\newcommand{\GL}{\operatorname{GL}}

\newcommand{\clmn}{c_{\lambda\mu}^\nu}
\newcommand{\dc}[2]{W_{#1} \backslash W / W_{#2}}

\newcommand{\cato}{\mathcal{O}}

\newcommand{\afsl}{\widehat{\mathfrak{sl}_2}}

\newcommand\sltwohat{\widehat{\mathfrak{sl}_2}}
\newcommand\slnhat{\widehat{\mathfrak{sl}_n}}
\newcommand\lieg{\mathfrak{g}}

\DeclareMathOperator\weight{wt}

\newcommand\myemptyset\Phi

\newcommand\vlambdanot{V(\Lambda_0)}
\newcommand\vlambdaone{V(\Lambda_1)}
\newcommand\chargedpart{\mathcal{P}}
\newcommand\tworegcpnot{{\chargedpart}^{\textup{2-reg}}_0}
\newcommand\tworegcpone{{\chargedpart}^{\textup{2-reg}}_1}

\newcommand\klwm{K(\lambda,w,\mu)}
\newcommand\klvm{K(\lambda,v,\mu)}

\newcommand\lie[1]{\mathfrak{#1}}
\newcommand\liegp{\mathfrak{g}}

\DeclareMathOperator{\ch}{char}

\newcounter{cnt}
 \makeatletter
\def\mydggeometry{\makeatletter\dg@YGRID=1\dg@XGRID=20\unitlength=0.003pt\makeatother}
\makeatother \theoremstyle{remark}

\numberwithin{equation}{section}
\makeatletter
\def\section{\def\@secnumfont{\mdseries}\@startsection{section}{1}%
  \z@{.7\linespacing\@plus\linespacing}{.5\linespacing}%
  {\normalfont\scshape\centering}}
\def\subsection{\def\@secnumfont{\bfseries}\@startsection{subsection}{2}%
  {\parindent}{.5\linespacing\@plus.7\linespacing}{-.5em}%
  {\normalfont\bfseries}}
\makeatother

\begin{document}
\title[Kostant-Kumar crystals]{Crystals for Kostant-Kumar modules of $\sltwohat$}
\author{Mrigendra Singh Kushwaha}
\address{Department of Mathematics, Faculty of Mathematical Sciences, University of Delhi, New Delhi, 110\,007, India}
\email{mrigendra154@gmail.com}
\author{K.~N.~Raghavan}
\address{The Institute of Mathematical Sciences, HBNI, Chennai, 600113, India. \hspace{20em} {.} 
Current address: School of Interwoven Arts and Sciences,
Krea University,  Sri City,  Andhra Pradesh 517646,  India}
\email{knr@imsc.res.in}
\author{Sankaran Viswanath}
\address{The Institute of Mathematical Sciences, HBNI, Chennai, 600\,113, India}
\email{svis@imsc.res.in}
\thanks{\noindent The second and third authors acknowledge support under a DAE project grant to IMSc.}
\begin{abstract}
    We consider the affine Lie algebra $\widehat{\mathfrak{sl}_2}$ and the Kostant-Kumar submodules of tensor products of its level 1 highest weight integrable representations. We construct crystals for these submodules in terms of the charged partitions model and describe their decomposition into irreducibles.
\end{abstract}
\maketitle

\mysection{Introduction}\mylabel{s:intro}
\noindent
Let $\lieg$ be a symmetrizable Kac-Moody algebra. Given a dominant integral weight $\lambda$ of $\liegp$, let $V(\lambda)$ denote the irreducible integrable $\liegp$-module with highest weight $\lambda$. For dominant integral weights $\lambda, \mu$, consider the tensor product decomposition:
\[ V(\lambda)\otimes V(\mu) = \bigoplus_\nu V(\nu)^{\oplus \clmn} \] 
where the sum runs over all dominant integral weights $\nu$, and $\clmn$ is the (finite) number of times $V(\nu)$ occurs in the decomposition.

There is a natural family $K(\lambda, w, \mu)$ of $\liegp$-submodules of $V(\lambda)\otimes V(\mu)$, indexed by elements $w$ in the Weyl group $W$ of $\lieg$. The $K(\lambda, w, \mu)$, called {\em Kostant-Kumar (KK) modules} in
\cite{krv}, are the cyclic submodules of the tensor product defined by:
\[ K(\lambda,w,\mu) = U\lie{g}(v_\lambda \otimes v_{w\mu}) \]
Here $U\lie{g}$ denotes the universal enveloping algebra of $\liegp$,
$v_\lambda$ a highest weight vector of $V(\lambda)$, and $v_{w\mu}$ a non-zero weight vector of the extremal weight space $w\mu$ of $V(\mu)$.

 Let $W_\lambda, \, W_\mu$ denote the stabilizers in $W$ of $\lambda, \mu$ respectively. 
It is easy to see that $\klwm$ depends only upon the double coset of $w$ in $\dc{\lambda}{\mu}$  (see, e.g., \cite[Section 5]{krv}). The $\klwm$ were first introduced by Kostant in order to formulate a refined version of the celebrated Parthasarathy-Ranga Rao-Varadarajan (PRV) conjecture. For further strengthenings of the PRV theorem, see \cite[Section 8]{krv}.

Since tensor products and their submodules are integrable modules in category $\cato$, we have complete reducibility \cite[Chapter 11]{kac} for $\klwm$, and we can write:
\[ K(\lambda, w, \mu)  = \bigoplus_\nu V(\nu)^{\oplus \clmn(w)}\]
The multiplicities $\clmn(w)$ have the following properties:
\begin{enumerate}
	\item $\clmn(w)\leq\clmn(v)$ if $W_\lambda w W_\mu \leq W_\lambda v W_\mu$ (Bruhat order on the double coset space). This follows from the containment relation (see, e.g., \cite[\S5]{krv}):
		\begin{equation}\label{eq:bruhatcont}	    
		\klwm\subseteq\klvm 
  \end{equation}
  \item $\clmn(w) \leq \clmn$ for all $w$, with equality holding for all ``sufficiently large'' $w$, i.e., for fixed $\lambda, \mu, \nu$, there exists $\sigma \in W$ such that equality holds for all $w \geq \sigma$. This is a consequence of the finiteness of $\clmn$ and Theorem~\ref{decompthm} below.
\end{enumerate}

In view of the second property, we sometimes think of the $\clmn(w)$ as a ``$w$-truncated" tensor product multiplicity.

When $\liegp$ is finite-dimensional or has a symmetric Cartan matrix (the latter assumption is expected to be unnecessary), Joseph \cite{joseph} (see also Lakshmibai-Littelmann-Magyar \cite{llm}) obtained a description of the multiplicities $\clmn(w)$ in terms of the path model or crystals. We state this below in the language of Lakshmibai-Seshadri (LS) paths. We recall that for a dominant integral weight $\mu$, an LS-path $\pi$ of shape $\mu$ is given by a sequence:
\[ \pi = (w_1>w_2>\ldots >w_r; \ 0 =a_0 < a_1 <a_2<\ldots <a_r=1) \]
where $w_i$  are minimal length coset representatives in $W/W_\mu$ and $a_i$ are rational numbers for $i=1,2,...,r$ satisfying certain integrality conditions (see Appendix below). We say that $w_1$ is the {\em initial direction}, $w_r$ is the {\em final direction} and $\weight \pi:=\sum_{i=1}^r (a_i - a_{i-1}) w_i \mu$ is the  weight (or endpoint) of the path $\pi$. For $0 \leq k \leq r$, we call $\sum_{i=1}^k (a_i - a_{i-1}) w_i \mu$ the {\em turning points} of $\pi$. Let $P_\mu$ denote the set of all LS-paths of shape $\mu$. Given a dominant integral weight $\lambda$, we say that $\pi \in P_\mu$ is {$\lambda$-dominant} if $\lambda + \gamma$ lies in the dominant Weyl chamber for every turning point $\gamma$ of $\pi$. With this preparation, we can now state:
\begin{theorem}[Joseph\cite{joseph}]\label{decompthm}
Let $\liegp$ be a symmetrizable Kac-Moody algebra that is either of finite type or has a symmetric Cartan matrix. Let $\lambda,\mu$ be dominant integral weights and $w$ be an element of the Weyl group. Then the decomposition of $K(\lambda,w,\mu)$ into irreducible $\liegp$-modules is: 
\begin{equation}
K(\lambda,w,\mu) = \bigoplus_{\pi \in P_\mu^\lambda(w)}V(\lambda+\weight \pi)
\end{equation}
where $P_\mu^\lambda(w)$ denotes the set of all $\lambda$-dominant LS-paths of shape $\mu$ whose initial direction is $\leq w$.
Equivalently, for all dominant integral weights $\nu$, we have: \[c_{\lambda \mu}^\nu(w) =\text{ the number of LS-paths }\pi \in P_\mu^\lambda(w) \text{ such that }\lambda+\weight \pi = \nu.\]
\end{theorem}

In \cite{krv}, we obtained an explicit path model for the module $K(\lambda, w, \mu)$. 
This was realized as a subset of the path model for the tensor product $V(\lambda) \otimes V(\mu)$. The latter is given by the set $P_\lambda * P_\mu$ of all concatenations of LS paths of shapes $\lambda$ and $\mu$. In this setting, by a path model  for $K(\lambda,w,\mu)$, we mean a subset $P(\lambda,w,\mu)$ of $P_\lambda * P_\mu$ that  (i) is invariant under the root (crystal) operators, (ii) has character coinciding with that of $K(\lambda,w,\mu)$ and (iii) satisfies the containment relations $P(\lambda,w,\mu) \subseteq P(\lambda,v,\mu)$ mirroring \eqref{eq:bruhatcont} for $W_\lambda w W_\mu \leq W_\lambda v W_\mu$.

Let $\pi_1$ and $\pi_2$ be LS-paths of shape $\lambda$ and $\mu$ respectively. We define $\mathfrak{w}(\pi_1 * \pi_2)$ to be the following element of the Weyl group \cite[\S 3, equation (1)]{krv}:
\begin{equation}\label{wasso}
\mathfrak{w}(\pi_1 * \pi_2): = \text{the unique minimal element of } W_\lambda \,I(\tau^{-1})\phi \,W_\mu
\end{equation}
where $\tau$ is the final direction of $\pi_1$, $\phi$ is the initial direction of $\pi_2$ and $I(\tau^{-1}) = \{\sigma \in W: \sigma \leq \tau^{-1}\}$ is the Bruhat ideal generated by  $\tau^{-1}$. We also define for $w \in W$:
\begin{equation}\label{eq:plwm}
P(\lambda,w,\mu):= \{ \pi_1 * \pi_2 \in P_\lambda * P_\mu : \mathfrak{w}(\pi_1 * \pi_2) \leq w \}
\end{equation}
We now have \cite[Theorem 7.1]{krv}:
\begin{theorem}\label{pathmod}
Let $\liegp$ be a symmetrizable Kac-Moody algebra that is either of finite type or has symmetric Cartan matrix. Let $\lambda,\mu$ be dominant integral weights and $w$ be an element of the Weyl group. The set $P(\lambda,w,\mu)$ is a path model for the Kostant-Kumar module $K(\lambda,w,\mu)$, i.e., it is invariant under the root operators and
\begin{equation}\label{eqpathmod}
\ch K(\lambda,w,\mu)= \sum_{\eta \in P(\lambda,w,\mu)} e^{\weight \eta}
\end{equation}
\end{theorem}

In this article, we apply these theorems to the smallest infinite dimensional Kac-Moody algebra $\afsl$. The simplest non-trivial highest weight integrable representations $V(\lambda)$ of $\afsl$ are those of level 1. The decomposition of the tensor products $V(\lambda) \otimes V(\mu)$ into irreducibles  was obtained by Kac-Peterson (Proposition~\ref{prdec} below) for all pairs $\lambda, \mu$ of level 1 using the Weyl-Kac character formula. Later, Misra-Wilson \cite{mw}  gave an alternate proof using the crystals of these representations, realized in terms of {\em charged partitions} (or coloured Young tableaux) (see \S \ref{deckk} below). 

The crystal of charged partitions gives a more tractable model for representations of affine Lie algebras of type A, in comparison to the crystal of LS paths. However, the latter is more amenable to the study of Kostant-Kumar modules in general. We use an explicit isomorphism between these models to obtain concise descriptions of Kostant-Kumar crystals in the language of charged partitions. More specifically, we establish the following main results concerning Kostant-Kumar modules $K(\lambda,w,\mu)$ for $\lambda, \mu$ of level 1: 
\begin{enumerate}
\item  An explicit formula in terms of charged partitions for the decomposition of $K(\lambda,w,\mu)$ into irreducibles. Our result turns out to give a natural $w$-truncation (by length of $w$) of the result of Misra-Wilson mentioned above (see Propositions \ref{kkd0} and \ref{kkd1}).
\item Crystals for the $K(\lambda,w,\mu)$, described as a subset of pairs of charged partitions with a compatibility condition (see Propositions \ref{kkcrys1} and \ref{kkcrys2}). This may be viewed as a generalization of the  crystal of multipartitions \cite{jimbo-et-al} for level 2 irreducible representations of $\widehat{\mathfrak{sl}_2}$  (which corresponds to the case $w=1$).
\end{enumerate}

\mysection{Preliminaries}\mylabel{s:prelims}
\noindent

We set up notations for use in the rest of the paper. Let  $\lie{sl}(2, \mathbb{C})$ denote the simple Lie algebra of $2\times2$ traceless matrices.  Let $E := E_{1,2}, F := E_{2,1}$ and $H = E_{1,1}-E_{2,2}$ be the Chevalley generators of $\lie{sl}(2, \mathbb{C})$, where $E_{i,j}$ denotes the matrix with a $1$ in the $i^{th}$ row and $j^{th}$ column, and $0$ elsewhere. The affine Lie algebra $\widehat{\lie{sl}_2} = \lie{sl}(2, \mathbb{C}) \otimes \mathbb{C}[t,t^{-1}] \oplus \mathbb{C}c \oplus \mathbb{C}d$ is infinite dimensional, and generated by the degree derivation $d$ and the Chevalley generators:
 \[ e_1 = E\otimes 1,\; f_1 = F\otimes 1,\;  h_1 = H\otimes 1,\; e_0 = F\otimes t,\; f_0 = E\otimes t^{-1},\; h_0 = -h_1 + c \]
 where $c = h_0+h_1$ is the canonical central element \cite[Chapter 7]{kac}. The Cartan subalgebra is $\lie h = {\mathbb{C}}$-span of $\{h_0,h_1,d \}$. Let $\delta = \alpha_0 + \alpha_1$ be the null root where $\alpha_0, \alpha_1$ are the simple roots of $\widehat{\lie{sl}_2}$.
Let $\Lambda_0, \Lambda_1 \in \lie h^*$ be the fundamental weights of $\widehat{\lie{sl}_2}$, satisfying $\Lambda_i(h_j) = \delta_{ij}$ for $i, j =1,2$. Let $P^+ = \mathbb{Z}_{\geq 0}$-span of $ \{\Lambda_0, \Lambda_1 \} +  \mathbb{C}\delta $ denote the set of dominant integral weights. Let $s_0, s_1 \in \GL(\lie h^*)$ denote the simple reflections corresponding to $\alpha_0, \alpha_1$ respectively. Then the Weyl group $W$ of $\widehat{\lie{sl}_2}$ is generated by $s_0,s_1$.
%

\mysection{$2$-regular charged partitions and their crystal structure}\mylabel{s:cyd}
\noindent
We recall two realizations of the crystal of the module $V(\Lambda_0)$ of $\afsl$ - in terms of (a) charged partitions (or coloured Young diagrams) \cite{jimbo-et-al,mm} and (b) Lakshmibai-Seshadri paths \cite{ys}. 
These two models are  isomorphic as crystals and we describe this isomorphism explicitly for use in later sections.

The set $\tworegcpnot$ (respectively $\tworegcpone$) of $2$-regular charged partitions of charge~$0$ (respectively, charge $1$) provides a combinatorial model for the fundamental representation $\vlambdanot$ (respectively $\vlambdaone$).   Recalled in this section are the basic definitions and results leading up to this fact, with full proofs (see, e.g., Remark~\ref{r:redsign}).     The root operators defined in~\S\ref{ss:rootops} endow $\tworegcpnot$ with its crystal structure.

The results in this section have generalizations to the case of $\slnhat$,  as is well known: see, e.g., \cite{mw}.
\mysubsectionstar{Notation}\noindent
Let $I=\{0,1\}$. Throughout this section, $i$ and $j$ are used to denote elements of~$I$.  
\mysubsection{Charged partitions and their diagrams}\mylabel{ss:cpdiag}
A {\em charged partition\/} is a pair $(\pi, j)$ where $\pi$ is an integer partition and $j$ is an element of~$I$.  We call $j$ the {\em charge\/}.
The {\em diagram\/} of a charged partition consists of the Young diagram of $\pi$,  the boxes of which are filled with elements of~$I$, such that $j-r+c  \pmod{2}$ is the entry in the box in the $r^{th}$ row (counted as usual from the top) and $c^{th}$ column (counted as usual from the left).  The entries in the boxes are sometimes also called {\em labels\/}.

We let $\myemptyset$ denote the \emph{empty Young diagram} and consider $(\myemptyset,0)$ and $(\myemptyset,1)$ as charged partitions.   
\mysubsection{A running example}\mylabel{ss:example}
For example, here is the diagram of the charged partition
$(8+6+3+1,0)$:
\begin{equation}\label{x:cyd}
	\begin{array}{cc}
		{		
	\begin{array}{|c|c|c|c|c|c|c|c|}
	\hline
	0 & 1 & 0 & 1 & 0 & 1 & 0 & 1\\
	\hline
	1 & 0 & 1 & 0 & 1 & 0 \\
	\cline{1-6}
	0 & 1 & 0\\
	\cline{1-3}
	1\\
	\cline{1-1}
\end{array}
		}{\begin{array}{rl}
			\textup{Weight:} & \Lambda_0-9\alpha_0-9\alpha_1\\
			\textup{$0$-signature:} & +\ +\ -\ -\ +\\	
			\textup{$1$-signature:} & -\ +\ +\ - \\	
			\textup{reduced $0$-signature:} & +\ +\ - \\	
			\textup{reduced $1$-signature:} &  +\ - \\	
		\end{array}
			}
		\end{array}
	\end{equation}

	\mysubsection{Weight}\mylabel{ss:weightcyd} 
	
	The {\em weight\/} of a charged partition $(\pi, j)$ is defined to be the element $\Lambda_j-n_0\alpha_0-n_1\alpha_1$ in the weight lattice of~$\lieg=\sltwohat$,
	where $n_0$ and $n_1$ are the numbers of zeroes and ones respectively in the diagram of the charged partition.
	The weight of the example in~\eqref{x:cyd} is $\Lambda_0-9\alpha_0-9\alpha_1$.

	\mysubsection{Removable and addable boxes; signatures}\mylabel{ss:sign}
A column of a charged partition is called {\em $i$-removable\/} if its bottom-most box is labelled $i$ and can be removed leaving another charged partition;
it is called {\em $i$-addable\/}, if a box labelled $i$ can be added to its bottom, giving another charged partition. 

The {\em $i$-signature\/} of a charge partition is a string of $+$ and $-$ signs, determined as follows. Scan the columns, one by one, from left to right (including the rightmost empty column).  Write a  $+$ if the column being scanned is $i$-addable, a $-$ if it is $i$-removable, and nothing if neither.   The string thus obtained is the $i$-signature.

For the example in~\eqref{x:cyd}, the $0$-addable columns are the first, second, and ninth;  the $0$-removable columns are the third and sixth; the $1$-addable columns are the fourth and seventh;  and the $1$-removable columns are the first and eighth.   The fifth column is neither $i$-addable nor $i$-removable for either value of~$i$.    
Thus the $0$-signature is $+$ $+$ $-$ $-$ $+$,
and the $1$-signature is $-$ $+$ $+$ $-$.

As we will see (in the definition of root operators in~\S\ref{ss:rootops}),  it will be important to keep track of which columns contributed to which signs in the $i$-signature.  For the example in~\eqref{x:cyd}, the numbers of the columns that contribute to the $0$-signature are $1$, $2$, $3$, $6$, $9$;   those that contribute to the $1$-signature are $1$, $4$, $7$, $8$.
\bremark\mylabel{r:redsign}   The {\em reduced $i$-signature\/} (of a charged partition), which we will formally define in~\S\ref{ss:sigred} below,  is usually defined as the string that is left behind after recursively deleting substrings of the form $-$ $+$ (that is,  a negative sign followed immediately by a positive sign)  from the $i$-signature.     For example,  the reduced $i$-signature is $+$ if the $i$-signature is $+$ $-$ $-$ $+$ $+$.    However it is not immediately clear from this description that different ways of deleting occurrences of $-$ $+$ lead to the same final string and, moreover,  that the status of a particular sign---whether or not it survives in the reduced signature---is independent of the chosen deletion order.

We therefore resort to a more formal definition,  by means of defining {\em reducible substrings\/} and making some observations about them (Lemmas~\ref{l:reducible} and \ref{l:reduction}).   The reader willing to accept the informal definition above may want to skip the next two subsections \S\ref{ss:reducible} and \S\ref{ss:sigred}.
\eremark

\mysubsection{Reducible substrings}\mylabel{ss:reducible}
A (finite) string, possibly empty, of minus and plus signs is said to be {\em reducible\/} if the following conditions hold:
\begin{enumerate}
	\item\label{i:redone} The number of minus signs equals the number of plus signs.
	\item\label{i:redtwo} Scanning the string from left to right,   the number of minus signs encountered up to any given point is at least the number of plus signs encountered up to that point.
\end{enumerate}
Here are two examples:   $-$ $+$  and $-$ $-$ $+$ $+$ $-$ $-$ $+$ $-$ $-$ $+$ $+$ $-$ $+$ $+$.

Note that any reducible string must start with $-$ and end with $+$.    Note also that, in the presence of condition~\ref{i:redone},  condition~\ref{i:redtwo} is equivalent to the following:
\begin{enumerate}
	\item[(2')] Scanning the string from right to left,   the number of minus signs encountered up to any given point is at most the number of plus signs encountered up to that point.
\end{enumerate}

Given a string of minus and plus signs,   a {\em substring\/} is any single continuous piece (possibly empty) of the string.   A substring that is reducible as a string in its own right is called a {\em reducible substring\/}.    For example, the string $-$ $-$ $+$ has $7$ substrings: \[\textup{the empty string, \quad $-$,\quad $+$,\quad $-$ $-$,\quad $-$ $+$,\quad $-$ $-$ $+$} \]
Of these,  only one, namely $-$ $+$, is reducible.
\blemma\mylabel{l:reducible}
\begin{enumerate}
	\item\label{i:lemone} Any initial substring of a string that satisfies condition~\ref{i:redtwo} also satisfies~\ref{i:redtwo}.
	\item\label{i:lemone'} If a reducible string is inserted into another reducible one (at any position), the resulting string is also reducible.  In particular, the concatenation of two reducible strings is reducible.
	\item\label{i:lemtwo} If a final substring~$v$ of a reducible string~$r$ satisfies condition~\ref{i:redtwo},  then the initial substring $r\setminus v$ is reducible (and so is~$v$).

	\item\label{i:lem3} Suppose that two reducible substrings either meet or occur immediately one after the other (meaning that there is no sign in between them). Then their union is a reducible substring.   In particular, maximal reducible substrings are separated (by signs occurring in between).
	\item\label{i:lem4} Suppose that a maximal reducible substring occurs in the middle,  meaning that it has both a sign that immediately precedes it and another that immediately follows it.  Then either the former sign is not $-$ or the latter is not $+$.   In other words,
		a string cannot have the following form:   $\ldots$ $-$ $m$ $+$ $\ldots$,   where $m$ is a maximal reducible substring.
	\item\label{i:lem5} A reducible substring that meets (non-trivially) a maximal reducible substring is contained in it.
	\item\label{i:lem6} A subset of a string that is a union of reducible substrings has an equal number of plus and minus signs.
\end{enumerate}
\elemma
\bmyproof  Items~\eqref{i:lemone} and~\eqref{i:lemone'} are immediate from the definition.   For~\eqref{i:lemtwo},  first observe that $r\setminus v$ satisfies condition~\ref{i:redtwo} by item~\eqref{i:lemone}. 
Now, since $r\setminus v$ and $v$ both satisfy~\ref{i:redtwo} and $r$ satisfies~\ref{i:redone}, it follows that both $r\setminus v$ and $v$ satisfy~\ref{i:redone}.

For \eqref{i:lem3}, let $r$ and $s$ be reducible substrings satisfying the hypothesis. Let their (possibly empty) intersection be~$w$ and suppose that $w$ is a final substring of~$r$ and an initial substring of~$s$.   
Now, $w$ satisfies~\ref{i:redtwo} by~\eqref{i:lemone},  so $r\setminus w$ is reducible by~\eqref{i:lemtwo}, and so $r\cup s=(r\setminus w)\cup s$ is reducible by~\eqref{i:lemone'}.

For~\eqref{i:lem4},  observe that $-$ $m$ $+$ is reducible (e.g. by~\eqref{i:lemone'}),  contradicting the maximality of~$m$.  Item~\eqref{i:lem5} is an immediate consequence of~\eqref{i:lem3}. Finally, for~\eqref{i:lem6}, observe that such a subset can be written,  using~\eqref{i:lem3}, as a disjoint union of reducible substrings and therefore the conclusion holds.~\emyproof

\mysubsection{Reduction of a string; reduced $i$-signature}\mylabel{ss:sigred}
Given a string of plus and minus signs,  its {\em reduction\/} is the string that remains after deleting the union of all reducible substrings.
The {\em reduced $i$-signature\/} of a charged partition is the reduction of its $i$-signature. 

From items~\eqref{i:lem3} and~\eqref{i:lem4} of Lemma~\ref{l:reducible},  it follows that the reduction always consists of some (possibly zero) plus signs followed by some (possibly zero) minus signs.
For the example in~\eqref{x:cyd},  
the reduced $0$-signature is $+$ $+$ $-$,
and the reduced $1$-signature is $+$ $-$.   The contributing column numbers to the $0$-signature and $1$-signature respectively are $1$, $2$, $3$ and $7$, $8$.
\blemma\mylabel{l:reduction}
Suppose that, after deleting from a string~$s$ the union of some reducible substrings, we are left with a string~$r$ consisting of some (possibly zero) plus signs followed by some (possibly zero) minus signs.     Then $r$ is the reduction of~$s$.
\elemma
\bmyproof   It is enough to show that there is no reducible substring of~$s$ that intersects~$r$ non-trivially.    Suppose, by way of contradiction, that $m$ is such a string.  We may assume that $m$ is maximal.

Suppose $m$ contains a plus sign belonging to~$r$.  (The proof in the case that $m$ contains a minus sign belonging to~$r$ is dually similar and therefore omitted.)  Choose a first (that is, left most) such sign that belongs to $m$, and let $m'$ be the initial string of~$m$ up to but not including this sign.  We allow for the possibility that $m'$ is empty.   

We claim that $m'$ has equal number of plus and minus signs.    It is enough to prove this,  for then, considering the initial substring $m'$ $+$ of~$m$, we see that $m$ fails condition~\ref{i:redtwo} in the definition of reducibility, and we have the desired contradiction.

To prove the claim,  let $s'$ be the initial string of~$s$ up to but not including the plus sign just after~$m'$.    We have $m'=(s'\setminus r)\cap m$.    It is enough, by item~\eqref{i:lem6} of Lemma~\ref{l:reducible}, to show that $m'$ is the union of reducible substrings of~$s$.

Let $x$ be any element of~$m'$.    Then $x\in s'\setminus r\subseteq s\setminus r$.   It follows from the hypothesis that there exists a reducible substring~$u$ of $s$ that contains $x$ and is contained in~$s\setminus r$.    Since $u$ is a substring,  it follows that $u\subseteq s'\setminus r$.  Since $u$ meets $m$ (they both contain~$x$),   it follows from item~\eqref{i:lem5} of Lemma~\ref{l:reducible} that $u\subseteq m$.   Thus $u\subseteq (s'\setminus r)\cap m=m'$.    Since $x$ was arbitrary,   this completes the proof.~\emyproof

\bremark\mylabel{r:signature}  Suppose that there are two consecutive plus (respectively minus) signs in the $i$-signature of a charged partition and that the one on the left (respectively right) survives in the reduced $i$-signature (or, in other words,  is not part of any reducible substring).   Then the other one also survives.   This is because a reducible substring 
cannot start with $+$ or end with $-$.
\eremark

\mysubsection{The functions $\varepsilon_i$ and $\varphi_i$}\mylabel{ss:phiepsilon}  
The function $\varepsilon_i$ (respectively $\varphi_i$) on charged partitions is defined to be the number of the minus (respectively plus) signs in the corresponding reduced $i$-signatures.
\mysubsection{Regular partitions}\mylabel{ss:regular}  Recall that a partition is called {\em $2$-regular\/} if each (non-zero) part of it occurs at most once.   In terms of the diagram associated to the partition this means that no two (non-empty) rows of the partition have the same number of boxes.  A charged partition is called {\em $2$-regular\/}  if the underlying partition is so.   Throughout this note, we say ``regular'' to mean ``2-regular''.    The example in~\eqref{x:cyd} is regular.

\bremark\mylabel{r:regular}  
Suppose that there is an $i$-addable (respectively $i$-removable) column of a regular charged partition and that the regularity is lost by adding (respectively removing) the box at the bottom of that column.   Then the next (respectively previous) column is also $i$-addable (respectively $i$-removable).
\eremark

\mysubsection{Root operators}\mylabel{ss:rootops}  
We recall the definition of the {\em root operators\/} $e_i$ and $f_i$ (for $i$ in~$I$) on charged partitions.    Augment the set of charged partitions with a special element, the {\em null element\/},  denoted $0$.   The $e_i$ (respectively $f_i$) acts on a charged partition by removing a box from (respectively adding a box to) the bottom of the column corresponding to the left most $-$ (respectively right most $+$) in the reduced $i$-signature.   In case there are no minus (respectively plus) signs in the $i$-reduced signature,  the operator~$e_i$ (respectively $f_i$) {\em kills\/} the charged partition, which just means that the action produces the null element $0$.

On the charged partition in the example in~\eqref{x:cyd},  the root operators act as follows:
\begin{equation*}
	\begin{array}{l}
		\textup{$e_0$ removes the bottom box labelled $0$ from the third column}\\
		\textup{$e_1$ removes the bottom box labelled $1$ from the eighth column}\\
		\textup{$f_0$ adds a box labelled $0$ to the bottom of the second column}\\
		\textup{$f_1$ adds a box labelled $1$ to the bottom of the fourth column}\\
	\end{array}
\end{equation*}
\blemma\mylabel{r:rootpreserve}  Let $b$ be a charged partition that is not killed by $e_i$ (respectively $f_i$) and put $b'=e_ib$ (respectively $b'=f_ib$).
\begin{enumerate}\item\label{i:lr1}   The charges of $b$ and $b'$ are the same.   In other words, charge is preserved by the root operators.
		\item\label{i:lr1'}  
			The weight of $b'$ equals $\weight{b}+\alpha_i$ (respectively $\weight{b}-\alpha_i$). 
		\item\label{i:lr1'a} Suppose that $r$ is a reducible substring of the $i$-signature of~$b$.   Then $r$ does not contain the contribution of the unique column that undergoes change under the passage from~$b$ to~$b'$,  and thus $r$ is also a reducible substring of the $i$-signature of~$b'$.
		\item\label{i:lr1''} The columns of $b$ that contribute to the reduced $i$-signature of~$b$ are precisely those that contribute to that of~$b'$ as well.
		The reduced signature of $b'$ is obtained from that of~$b$ by changing the left most $-$ in it to $+$ (respectively changing the right most $+$ in it to~$-$). 
	\item\label{i:lr1''a} $f_ib'=b$ (respectively $e_ib'=b$).  Thus $e_i$ and $f_i$ are ``partial'' inverses of each other.
		\item\label{i:lr'''} The value of the function~$\varepsilon_i$ (respectively~$\varphi_i$) on a charged partition equals the maximum number of times $e_i$ (respectively $f_i$) can be applied to it without killing it.
		\item\label{i:lr2} If $b$ is regular then so is~$b'$. In other words,  the set of regular charged partitions,  along with the null element~$0$,  is stable under the action of the root operators.
\end{enumerate}
\elemma
\bmyproof
			Items~\eqref{i:lr1} and~\eqref{i:lr1'} are immediate from the definitions. Item~\eqref{i:lr1'a} is also clear since the unique column that changes contributes to the reduced $i$-signature of~$b$.

			The second statement in~\eqref{i:lr1''} follows from the first.   To prove the first,  let $r$ be the union of all reducible substrings of the $i$-signature of~$b$.  By~\eqref{i:lr1'a}, $r$ is also a union of reducible substrings of~$b'$.    Removing $r$ from the $i$-signature of~$b'$ leaves a string of plus signs followed by minus signs.   It follows from Lemma~\ref{l:reduction} that this remainder is the reduced $i$-signature of~$b'$ and that $r$ is also the union of all reduced substrings of~$b'$.

			Item~\eqref{i:lr1''a} follows from~\eqref{i:lr1''} and~\eqref{i:lr'''} in turn from~\eqref{i:lr1''a}.   Item~\eqref{i:lr2} follows from Remarks~\ref{r:signature} and \ref{r:regular}.
			\emyproof
\ignore{
has above example \ref{excyd}, the $0$-\emph{signature} is $\{++--+\}$ and reduced $0$-\emph{signature} is $\{++-\}$. We define $\widetilde{f_i} b$ to be the charged partition obtained from $b$ by adding a box to the bottom of the column corresponding to the rightmost $+$  in its reduced $i$-signature. If there is no $+$ sign in the reduced $i$-\emph{signature}, we define $\widetilde{f_i} b =0$. Analogously, $\widetilde{e_i}$ acts on $b$ by removing the bottom-most box of the column corresponding to the leftmost $-$. If there is no $-$ sign in the reduced signature, we set $\widetilde{e_i} b =0$.  The function $\varphi(b)$ is the number of $+$ signs and $\varepsilon(b)$ the number of $-$ signs in the reduced $i$-\emph{signature} of $b$. The weight function $wt:\mathcal{Y}(0)\rightarrow P$ is defined by $b\rightarrow \Lambda_0 - \sum_{i=0}^{1} \#\{i-coloured \ box \ in \ b \}\, \alpha_i$. Then $\mathcal{Y}(0)\cong B_{\Lambda_0}$ \cite{jimbo-et-al,mm}.

Similarly, we have the set $\mathcal{Y}(1)$ of charged partitions of charge $1$. Define $e_i, f_i$ correctly here. Then $\mathcal{Y}(1)\cong B_{\Lambda_1}$.

%
}
\subsection{LS paths for $\widehat{\lie{sl}_2}$}\label{lspathsl2}
We will assume all basic notation and results from Littelmann's theory of paths \cite{litt:inv,litt:ann}. A path $\pi$ is a piecewise linear map $\pi:[0,1]\rightarrow \lie{h}^*$. Let $\Pi$ be the collection of all paths such that for $\pi \in \Pi$, $\pi(0)=0$ and $\pi(1)$ is an integral weight. We identify two paths $\pi, \pi'$ in $\Pi$ if their images in $\lie h^*$ are the same. 
 
   Let $\lambda \in P^+$ be a dominant integral weight. A Lakshmibai-Seshadri path (LS path) $\pi$ of shape $\lambda$ is given by a sequence:
\begin{equation}
\pi= (\bar{w},\bar{a}): = (w_1>w_2>...>w_r; 0<a_1<a_2<...<a_r=1 ).
\end{equation}
satisfying some integrality conditions (see the Appendix below) \cite{litt:inv}. Here $w_j$ are minimal coset representatives in $W/W_\lambda$ and $a_j$ are rational numbers for $j=1,2,\ldots,r$.  Let $P_\lambda $ denote the set of all LS paths of shape $\lambda$. Then $P_\lambda$ can be given a crystal structure with the root operators $e_i, f_i$ defined for every simple root $\alpha_i$ \cite{litt:ann}.  

\subsection{}
We now recall Sanderson's  description \cite{ys} of LS paths for the level 1 representations $V(\Lambda_0)$ and $V(\Lambda_1)$ of $\widehat{\mathfrak{sl}_2}$. First, note that
\begin{equation}\label{stb0}
W/W_{\Lambda_0}  = \{ w_n^+ := s_{i_n}\cdots s_{i_2}s_{i_1}: n \geq 0,\; i_j \equiv j+1\pmod{2}\}
\end{equation}
\begin{equation}\label{stb1}
W/W_{\Lambda_1}  = \{ w_n^- := s_{i_n} \cdots s_{i_2}s_{i_1}: n \geq 0,\; i_j\equiv j\pmod{2}  \}
\end{equation}
 where $W_{\Lambda_0}$ and $W_{\Lambda_1}$ are the stabilizers in $W$ of $\Lambda_0,\Lambda_1 $ respectively. We observe that the Bruhat order on these coset spaces is a  total order:
  \begin{equation} \label{eq:bruhat}
  w_m^\epsilon \geq w_n^\epsilon \text{ if and only if } m \geq n, \text{ for } \epsilon \in \{+,-\}
\end{equation}
 We also note that if $\alpha_0^\vee, \alpha_1^\vee$ denote the simple coroots, then
 \begin{equation}\label{aglev}
 \begin{array}{cc}
 \langle w_{2k}^+\,\Lambda_0, \,\alpha_0^\vee \rangle = 2k+1;  &  \langle w_{2k}^+\,\Lambda_0, \,\alpha_1^\vee \rangle = -2k \\
 \langle w_{2k-1}^+ \Lambda_0, \,\alpha_0^\vee \rangle = -(2k-1);  &  \langle w_{2k-1}^+\Lambda_0, \,\alpha_1^\vee \rangle = 2k
 \end{array}
 \end{equation}
 where $\langle \cdot, \cdot \rangle$ is the pairing $\mathfrak{h}^* \times \mathfrak{h} \to \mathbb{C}$.
 
\subsection{}
We primarily consider LS paths of shape $\Lambda_0 $, because of the symmetric role of $\alpha_0$ and $\alpha_1$. In the remainder of the paper, we will let $w_i:=w_i^+$ to declutter notation. 

\begin{lem}\cite[\S 4.1]{ys}\label{lspathdes}
     A path $\pi = (\bar{w},\bar{a})$ is an LS path of shape $\Lambda_0$ if and only if
\begin{equation}\label{eq:wseq}
\bar{w}:\; w_{m} > w_{m-1} > \cdots >w_{n+1} > w_{n}  \;\;\;\; (m\geq n \geq 0)
\end{equation}
\begin{equation}\label{ras}
\bar{a} :\; 0 < \frac{i_m}{m} < \frac{i_{m-1}}{m-1}< \cdots < \frac{i_{n+1}}{n+1}<1
\end{equation}  
where $i_j\in\{ 1,2,...,j-1 \}$ for $n+1 \leq j \leq m$. 
Further, the tuple $(i_m,i_{m-1},...,i_{n+1})$ in \eqref{ras} satisfies the above inequalities if and only if:
\begin{equation}\label{lsinq}
 1 \leq i_m \leq i_{m-1} \leq ... \leq i_{n+1} \leq n.
\end{equation}
\end{lem}
It is important to note here that the Weyl group elements involved form a contiguous sequence, with no skips.

\begin{rem}\label{rem} The LS paths of shape $\Lambda_1$ also have a similar description - we only have to replace $w_i= w_i^+$ by $w_i^-$  in \eqref{eq:wseq}.
\end{rem}

\subsection{Crystal isomorphism}\label{s:cryiso}
 We will define a map $\Psi$ from the set of $2$-\emph{regular} charged partitions $\mathcal{Y}(0)$ to the set of LS paths $P_{\Lambda_0}$ of shape $\Lambda_0$. Let  $b = (\mu_1 > \mu_2 > \cdots >\mu_l >0) \in \mathcal{Y}(0)$  be a partition into distinct parts (i.e., $2$-\emph{regular}); thus $\mu_1 \geq l$. The dimensions $m \times n$ of the smallest rectangle that bounds $b$ are thus $m = \mu_1$ and $n=l$. Associate a partition $\widetilde{b}=(\mu_1-n, \mu_2-(n-1),...,\mu_l-1)$. Its conjugate partition $\widetilde{b}'$ has exactly $\mu_1 -n = m-n$ parts, which we denote $(i_{n+1},i_{n+2},...,i_m)$. Now define the LS path $\Psi(b) = (\bar{w},\bar{a})$ as follows:
 \begin{equation}
\bar{w}: w_{m} > w_{m-1} > ... >w_{n+1} > w_{n}
\end{equation}
\begin{equation}
\bar{a} :  0 < \frac{i_m}{m} < \frac{i_{m-1}}{m-1}<...< \frac{i_{n+1}}{n+1}<1.
\end{equation}
Then $\Psi$ is well defined. Clearly $m\geq n\geq 0$, and $i_m \geq 1$. Since 
\[i_{n+1} = \#\{1 \leq j \leq l: \mu_j-(n-j+1) \geq 1\},\] we conclude $i_{n+1}\leq l=n$. Hence we see that the $(m-n)$ tuple $( i_{m},i_{m-1},\ldots,i_{n+1})$ satisfies the inequalities  $1\leq i_{m} \leq i_{m-1} \leq ...\leq i_{n+1} \leq n $. By the discussion preceding \eqref{lsinq}, we conclude that $\Psi(b)$ is indeed an LS path of shape $\Lambda_0$. Note that if $b=\Phi$ (the empty partition) then $\Psi(b) = (1;0<1)$ is the straight line path from $0$ to $\Lambda_0$.

It is easy to see (using \eqref{lsinq}) that the map $\Psi$ is a bijection; the inverse map can be easily defined.
%
The key property of $\Psi$ is the following:
\begin{theorem}\label{cryiso}
The map $\Psi: \mathcal{Y}(0) \rightarrow P_{\Lambda_0} $ is a crystal isomorphism.
\end{theorem}

\begin{rem}\label{cryisorem}
There is also a crystal isomorphism $\Psi': \mathcal{Y}(1) \rightarrow P_{\Lambda_1}$, defined analogously to $\Psi$.
\end{rem}
The proof of Theorem~\ref{cryiso} appears in the appendix.

\begin{rem}\label{prop:takeaway}
The key takeaway from this isomorphism is this: the initial and final directions of the LS path are $w_m, w_n$ where $m, n$ are the dimensions of the bounding rectangle of the charged partition.
\end{rem}

\section{Decomposition of Kostant-Kumar modules}\label{deckk}
We recall the definition of Kostant-Kumar modules from the introduction.
Let $\lambda, \mu$ be dominant integral weights and $w$ be a Weyl group element of a symmetrizable Kac-Moody algebra $\lie g$. Let $v_\lambda$ be a highest weight vector of $V(\lambda)$ and $v_{w\mu}$ be a non-zero weight vector of weight $w\mu$ of $V(\mu)$. The \emph{Kostant-Kumar module} denoted by $K(\lambda,w,\mu) = U\lie{g}(v_\lambda \otimes v_{w\mu})$ is the cyclic submodule of $V(\lambda)\otimes V(\mu)$ generated by the vector $v_\lambda \otimes v_{w\mu}$.

\subsection{Decomposition of tensor products for $\widehat{\lie{sl}_2}$}\label{dectp}
  For the affine Lie algebra $\widehat{\lie{sl}_2}$, we will recall the explicit decomposition of tensor products of level 1 modules, and the generating function of outer multiplicities given in Kac \cite{kac} and Misra-Wilson \cite{mw}. Let $\Lambda_i$ for $i=0,1$ be the fundamental weights of $\widehat{\lie{sl}_2}$.
  
\begin{prop}\label{prdec}

 The tensor product $V(\Lambda_i)\otimes V(\Lambda_0)$  for $i =0,1$ decomposes as follows into irreducibles:
\begin{equation}\label{deceq1}
V(\Lambda_0)\otimes V(\Lambda_0) = \sum_{n \geq 0} a_n V(2\Lambda_0-n\delta)+ \sum_{n \geq 0} b_n V(2\Lambda_0-\alpha_0-n\delta)
\end{equation}
 where the outer multiplicities $a_n, b_n$ are given by following generating function:
 \begin{equation}\label{gen1}
 \sum_{n\geq 0} a_n x^{2n} + \sum_{n\geq 0} b_n x^{2n+1} = \prod_{\substack{j \geq 1 \\ j \text{ odd }}}(1+x^{j})
 \end{equation}
 Similarly, the other tensor product decomposes thus: 
\begin{equation}\label{deceq2}
V(\Lambda_1)\otimes V(\Lambda_0) = \sum_{n \geq 0} a_n V(\Lambda_0+\Lambda_1 -n\delta)
\end{equation}
 where the $a_n$ are given by following generating function:
 \begin{equation}\label{gen2}
 \sum_{n\geq 0} a_n x^{2n}  = \prod_{\substack{j \geq 1 \\ j \text{ even }}}(1+x^{j})
 \end{equation}
 \end{prop}
We remark that while these were originally proved using the Kac-Weyl character formula, Misra-Wilson employ the crystal of charged partitions to obtain a new proof. It is this latter proof that generalizes to the case of Kostant-Kumar modules, as we describe next.

\subsection{}
 For each Weyl group element $w$, we now compute the outer multiplicities of the Kostant-Kumar modules $K(\Lambda_i,w,\Lambda_0)$ for $i=0,1$. The corresponding generating functions turn out to be truncations of the generating functions \eqref{gen1} and \eqref{gen2} depending on the length of $w$.
  
  When $\ell(w)=0$, i.e., $w=1$, we have:
  \begin{equation}
  K(\Lambda_0,1,\Lambda_0) \cong V(2\Lambda_0)
  \end{equation}
  
  \begin{equation}
  K(\Lambda_1,1,\Lambda_0) \cong V(\Lambda_0 + \Lambda_1) \cong K(\Lambda_1,s_0,\Lambda_0)
  \end{equation}
  
  In the following two propositions, we will only take $w=w_m^+$. This suffices because for $m>0$,
  $w_m^- \in w_{m-1}^+ W_{\Lambda_0}$ and so $K(\Lambda_0,w_m^-,\Lambda_0) = K(\Lambda_0,w_{m-1}^+,\Lambda_0)$.
Recall here again that $w_m:=w_m^+$. 
 \begin{prop}\label{kkd0}
Let $m>0$, then the decomposition of  $K(\Lambda_0,w_m,\Lambda_0)$ is:
\begin{equation}\label{kdeceq1}
K(\Lambda_0,w_m,\Lambda_0) = \sum_{n \geq 0} a_n V(2\Lambda_0-n\delta)+ \sum_{n \geq 0} b_n V(2\Lambda_0 - \alpha_0 -n\delta)
\end{equation}
where the multiplicities $a_n, b_n$ are given by the generating function:
\begin{equation}\label{kgen1}
    \sum_{n\geq 0} a_n x^{2n} + \sum_{n\geq 0} b_n x^{2n+1} = \prod_{\substack{1 \leq \, j \leq \, m \\ j \text{ odd }}} (1+x^{j})
\end{equation}
\end{prop}
 
 \begin{prop}\label{kkd1}
Let $m>1$, then the decomposition of  $K(\Lambda_1,w_m,\Lambda_0)$ is:
\begin{equation}\label{kdeceq2}
K(\Lambda_1,w_m,\Lambda_0) = \sum_{n \geq 0} a_n V(\Lambda_0+\Lambda_1 -n\delta)
\end{equation}
 where the multiplicities $a_n$ are given by the generating function:
 \begin{equation}\label{kgen2}
 \sum_{n\geq 0} a_n x^{2n} = \prod_{\substack{1 \leq \, j \leq \, m \\ j \text{ even }}} (1+x^{j})
 \end{equation} 
 \end{prop}
 
 To prove propositions \ref{kkd0} and \ref{kkd1}, we will use the crystal isomorphism between the set of $2$-\emph{regular} charged partitions $\mathcal{Y}(0)$ and the set of LS paths $P_{\Lambda_0}$ of shape $\Lambda_0$. 

\subsection{Proof}\label{kkdcmp}
We recall some basic notation and results from \cite{mw}. Given $b \in \mathcal{Y}(0)$, we say $b$ is $\Lambda_j$-dominant for $j\in \{0,1\}$ if $e_i^{\delta_{ij}+1}(b)=0$ for $i\in \{0,1\}$ where $\delta_{ij}$ is the Kronecker delta function. The description of $\Lambda_j$-dominant elements $b \in \mathcal{Y}(0)$ was given in \cite[Lemma 3.3]{mw} for all $\widehat{\lie{sl}_n}$. We will state here only the $\widehat{\lie{sl}_2}$ version, which suffices for our requirements.

\begin{lem}\label{l0domb} (Misra-Wilson) 
\textbf{a.} Let $b \in \mathcal{Y}(0)$, then $b$ is $\Lambda_0$-dominant if and only if $b$ is a partition into distinct odd parts.\\
\textbf{b.}  Let $b \in \mathcal{Y}(0)$, then $b$ is $\Lambda_1$-dominant if and only if $b$ is a partition into distinct even parts.
\end{lem}

Let $\mathcal{Y}(0)^{\Lambda_0}$ denote the collection of all partitions into distinct odd parts. Given $b \in \mathcal{Y}(0)^{\Lambda_0}$, we can determine its weight as follows:
\begin{equation}\label{wtb}
wt(b) = \left\lbrace \begin{array}{cc}
       \Lambda_0 - k\delta & \text{ if }  b \vdash 2k \\
       \Lambda_0 - k\delta - \alpha_0 & \text{ if }  b \vdash 2k+1

\end{array} \right.
\end{equation} 

Likewise, let $\mathcal{Y}(0)^{\Lambda_1}$ be the collection of all partitions into distinct even parts. For $b \in \mathcal{Y}(0)^{\Lambda_1}$, its weight is given by:
\begin{equation}\label{wtbb}
wt(b) = \Lambda_0 - k\delta \text{ if } b\vdash 2k.
\end{equation}

\subsection{}
We now consider the LS path model side of things. The decomposition rule for Kostant-Kumar modules was given in \cite{krv} in terms of LS paths. We will recall here some basic notation and results. Let $\pi$ be an LS path of shape $\mu$
\begin{equation}\label{kkdecomps}
\pi = \left\lbrace\begin{array}{c}
 \sigma_1>\sigma_2>\cdots >\sigma_r \\
0<a_1<a_2<\cdots <a_{r-1}<a_r=1. \end{array}\right.
 \end{equation}
Thus $\sigma_1$ is the {initial direction}, $\sigma_r$ is the {final direction} and $\weight \pi:=\sum_{i=1}^r (a_i - a_{i-1}) \sigma_i \mu$ is the  weight of the path $\pi$. For $0 \leq k \leq r$, recall that $\sum_{i=1}^k (a_i - a_{i-1}) \sigma_i \mu$ are the {\em turning points} of $\pi$. We say that $\pi$ is {$\lambda$-dominant} if $\lambda + \gamma$ lies in the dominant Weyl chamber for every turning point $\gamma$ of $\pi$.  Let $P_\mu^\lambda(\sigma)$ denote the set of all $\lambda$-dominant LS-paths of shape $\mu$ whose initial direction is $\leq \sigma$.

Now, by theorem ~\ref{decompthm}, the decomposition of $K(\Lambda_i,w_m,\Lambda_0)$ for $i\in \{0,1\}$ is given by:
\begin{equation}\label{eqkkdecl0}
K(\Lambda_i,w_m,\Lambda_0) = \bigoplus_{\pi \in P_{\Lambda_0}^{\Lambda_i}(w_m)} V(\Lambda_i+\weight\pi).
\end{equation}

Now we will prove propositions \ref{kkd0} and \ref{kkd1} using the decomposition \eqref{eqkkdecl0}, rewritten in terms of \emph{charged partitions}.  For $i\in \{0,1\}$, define sets
\begin{equation}\label{l0bm}
\mathcal{Y}(0)^{\Lambda_i}(m):= \{b=(\mu_1,\mu_2,...,\mu_l) \in  \mathcal{Y}(0)^{\Lambda_i}| \ \mu_1 \leq m\}
\end{equation}
Note that from the definition of the root operators $e_i$ on paths (see Appendix), it follows that $\pi$ is $\Lambda_j$-dominant for $j\in \{0,1\}$ iff $e_i^{\delta_{ij}+1}(b)=0$ for $i\in \{0,1\}$. It is now clear from Remark~\ref{prop:takeaway}, Lemma~\ref{l0domb} and the fact that $\Psi$ is a crystal isomorphism that
\begin{equation}\label{eq:cp}
  \Psi^{-1}(P_{\Lambda_0}^{\Lambda_i}(w_m)) = \mathcal{Y}(0)^{\Lambda_i}(m) \text{ for } i\in \{0,1\}
\end{equation}
\textbf{Proof of Proposition \ref{kkd0}}: Equations \eqref{eqkkdecl0} and \eqref{eq:cp} imply that:
\begin{equation}\label{kkdcyd}
K(\Lambda_0,w_m,\Lambda_0) = \bigoplus_{b \in \mathcal{Y}(0)^{\Lambda_0}(m)} V(\Lambda_0+ wt(b)).
\end{equation}
For $b \in \mathcal{Y}(0)^{\Lambda_0}(m)$, its weight is given by \eqref{wtb}:
\begin{equation}\label{wtl0b}
\Lambda_0 + wt(b) = \left\lbrace \begin{array}{cc}
       2\Lambda_0 - k\delta & \text{ if } b \vdash 2k \\
       2\Lambda_0 - \alpha_0 -k\delta & \text{ if } b \vdash 2k+1

\end{array} \right.
\end{equation}
For $k \geq 0$, let 
\begin{align*} a_k&:=\#\{b \in \mathcal{Y}(0)^{\Lambda_0}(m): b \vdash 2k\} \\b_k&:=\#\{b \in \mathcal{Y}(0)^{\Lambda_0}(m): b \vdash 2k+1\} \end{align*}  
 Then the decomposition \eqref{kkdcyd} can be written as follows:
 \begin{equation}\label{kdeceq3}
K(\Lambda_0,w_m,\Lambda_0) = \sum_{k \geq 0} a_k V(2\Lambda_0-k\delta)+ \sum_{k \geq 0} b_k V(2\Lambda_0-\alpha_0 -k\delta)
\end{equation}
If $m>0$, it follows readily from \eqref{l0bm} that the generating function for $a_k$ and $b_k$ is
\begin{equation*}
 \sum_{k\geq 0}a_k x^{2k} + \sum_{k\geq 0}b_k x^{2k+1} = \prod_{\substack{1 \leq \, j \leq \, m \\ j \text{ odd }}} (1+x^{j})
\end{equation*}\qed

\noindent
\textbf{Proof of Proposition \ref{kkd1}}: Equation \ref{eqkkdecl0} gives:
\begin{equation}\label{kkdcyd2}
K(\Lambda_1,w_m,\Lambda_0) = \bigoplus_{b \in \mathcal{Y}(0)^{\Lambda_1}(m)} V(\Lambda_1+ wt(b)).
\end{equation}
As above, for $b \in \mathcal{Y}(0)^{\Lambda_1}(m)$, we have by equation \eqref{wtbb}:
\begin{equation}\label{wtl0b2}
\Lambda_1 + wt(b) = \Lambda_0+\Lambda_1 - k\delta; \text{ if } b\vdash 2k.
\end{equation}
 Let $a_k:=\#\{b \in \mathcal{Y}(0)^{\Lambda_1}(m)| b \vdash 2k\}$. 
 Then \eqref{kkdcyd} can be written as follows:
 \begin{equation}\label{kdeceq4}
K(\Lambda_1,w_m,\Lambda_0) = \sum_{k \geq 0} a_k V(\Lambda_0+\Lambda_1-k\delta)
\end{equation}
If $m>1$, then \eqref{l0bm} implies that the generating function for $a_k$ is:
\begin{equation*}
 \sum_{k\geq 0}a_k x^{2k} = \prod_{\substack{1 \leq \, j \leq \, m \\ j \text{ even }}} (1+x^{j})
\end{equation*}\qed

\section{Crystals for Kostant-Kumar  modules}

\subsection{}
Given $b \in \mathcal{Y}(i)$, let $m(b)$ and $n(b)$ denote the horizontal and vertical lengths respectively of the bounding rectangle of $b$ for $i=0,1$. In other words, if  $b = (\mu_1,\mu_2,...,\mu_l)$, then  $m(b) = \mu_1$ and $n(b)=l$; we have $m(b) \geq n(b)$.

Let $p$ be $0$ or an odd number, the element $w^+_p$ is the minimal representative of its double coset $W_{\Lambda_0} w^+_p W_{\Lambda_0}$. For $p \neq 0$, let $\mathcal{K}_p$ denote the set of all pairs $(b_1, b_2)$ in $\mathcal{Y}(0)^2$ such that $m(b_2) - n(b_1) \leq p+1$. We define $\mathcal{K}_0$ to be the set of pairs satisfying $m(b_2) \leq n(b_1)$. 

\begin{prop}\label{kkcrys1}
 The set $\mathcal{K}_p$ is a crystal for $K(\Lambda_0,w_p^+,\Lambda_0)$. In other words, $\mathcal{K}_p$ is invariant under the $e_i, f_i$ and its character equals that of $K(\Lambda_0,w^+_p,\Lambda_0)$.
\end{prop}

Let $p$ be $0$ or an even number, the element $w^+_p$ is the minimal representative of its double coset $W_{\Lambda_1} w^+_p W_{\Lambda_0}$. Let $\mathcal{K}_p'$ denote the set of all pairs $(b_1, b_2)$ in $\mathcal{Y}(1) \times \mathcal{Y}(0)$ such that $m(b_2) - n(b_1) \leq p+1$.

\begin{prop}\label{kkcrys2}
 The set $\mathcal{K}_p'$ is a crystal for $K(\Lambda_1,w_p^+,\Lambda_0)$. In other words, $\mathcal{K}_p'$ is invariant under the ${e}_i, {f}_i$ and its character equals that of $K(\Lambda_1,w^+_p,\Lambda_0)$.
\end{prop}

Proofs of the above propositions \ref{kkcrys1} and \ref{kkcrys2} will be deduced by using theorem \ref{pathmod} and the crystal isomorphisms $\Psi \times \Psi$ and $\Psi' \times \Psi$. In particular, we only need to determine the set $P(\lambda, w, \mu)$ (equation~\eqref{eq:plwm}) for the two cases above, i.e., for $(\lambda, w, \mu) = (\Lambda_0,w^+_p,\Lambda_0)$ and $(\lambda, w, \mu) = (\Lambda_1,w^+_p,\Lambda_0)$. Towards this end, we state the following lemmas:
\begin{lemma}\label{lem:min-lem-1}
  Let $\phi=w_m^+$ and $\tau = w_n^+$ with $m, n \geq 0$. Then \[\min\, W_{\Lambda_0} \, I(\tau^{-1})\phi \,W_{\Lambda_0} = w_{\ell}^+\] where
  \[ \ell = \begin{cases} \max(0,m-n-1) & \text{ if } m \equiv n \pmod{2} \\
    \max(0, m-n)  & \text{ if } m \not\equiv n \pmod{2}
    \end{cases}
      \]
\end{lemma}
\begin{lemma}\label{lem:min-lem-2}
  Let $\phi=w_m^+$ and $\tau = w_n^-$ with $m, n \geq 0$. Then \[\min\,W_{\Lambda_1} \,I(\tau^{-1})\phi \,W_{\Lambda_0} = w_{\ell}^+\] where
  \[ \ell = \begin{cases} \max(0,m-n) & \text{ if } m \equiv n \pmod{2} \\
    \max(0, m-n-1)  & \text{ if } m \not\equiv n \pmod{2}
    \end{cases}
      \]
\end{lemma}

\subsection{} To prove these lemmas, we recall from \cite{krv} the following procedure to compute
$\min I(x) y$ for $x, y \in W$. Given $w \in W$ and a simple reflection $s$, let $w \wedge sw$ denote the minimum of $\{w, sw\}$ in Bruhat order. Let $x = s_{i_k} s_{i_{k-1}} \cdots s_{i_1}$ be a reduced expression, and define inductively:
\begin{equation}\label{eq:zseq}
  z_1 = y \wedge s_{i_1} y \text{ and } z_j = z_{j-1} \wedge s_{i_j} z_{i_{j-1}} \text{ for } 2 \leq j \leq k
\end{equation}
Then $\min I(x) y = z_k$. Further, $\min W_\lambda I(x) y W_\mu = \min W_\lambda z_k W_\mu$ for any dominant integral weights $\lambda, \mu$ \cite[Cor 2.5, 2.6]{krv}.

\subsection{} For Lemma~\ref{lem:min-lem-1}, we have $\phi = s_{m-1}  \cdots s_1 s_0$ and $\tau = s_{n-1}  \cdots s_1 s_0$ where we define $s_j:=s_0$ if $j$ is even and $s_1$ if $j$ is odd. Let $x = \tau^{-1}$ and $y =\phi$.

If $m$ and $n$ are of the same parity, then $s_{m-1} = s_{n-1}$  and we have $z_1 = s_{m-1} \phi = w_{m-1}^+$ in \eqref{eq:zseq}. Likewise, $z_2 = s_{m-2} s_{m-1} \phi = w_{m-2}^+$ and so on, until $z_n = w_{d}^+$ where $d=0$ if $m \leq n$ and $m-n$ otherwise. In the latter case, $d$ is even and we obtain $\min W_{\Lambda_0} w_d^+ W_{\Lambda_0} = \min W_{\Lambda_0} w_{d-1}^+ W_{\Lambda_0}$.  This establishes the first case of Lemma~\ref{lem:min-lem-1}.

Next, if $m$ and $n$ are of opposite parity, then $s_{m-1} \neq s_{n-1}$. In this case $z_1 = \phi = w_{m}^+$ in \eqref{eq:zseq}. But now $s_{m-1} = s_{n-2}$ and so $z_2 = s_{m-1} \phi = w_{m-1}^+$ and so on, until $z_n = w_{d}^+$ where $d=0$ if $m \leq n-1$ and $m-n+1$ otherwise. Again in the latter case, $d$ is even and we obtain $\min W_{\Lambda_0} w_d^+ W_{\Lambda_0} = \min W_{\Lambda_0} w_{d-1}^+ W_{\Lambda_0}$.  This establishes the second case of Lemma~\ref{lem:min-lem-1}. 

The proof of Lemma~\ref{lem:min-lem-2} is analogous, taking into account the fact that
\[ \min W_{\Lambda_1} w_d^+ W_{\Lambda_0} = \min W_{\Lambda_1} w_{d-1}^+ W_{\Lambda_0} \]
when $d$ is odd. \qed

\subsection{} We prove Proposition~\ref{kkcrys1} first. 
Given $(b_1,b_2) \in \mathcal{Y}(0)^2$, let $m(b_i), n(b_i)$ denote the dimensions of the bounding rectangles, for $i=1, 2$. Let $\phi = w_m^+$ and $\tau = w_n^+$ where $m = m(b_2)$ and $n=n(b_1)$. 
We recall the crystal isomorphism $\Psi$ from Theorem~\ref{cryiso}. It is now clear from Remark~\ref{prop:takeaway} and Equation \eqref{eq:bruhat} that $(\Psi(b_1), \Psi(b_2)) \in P(\Lambda_0,w^+_p,\Lambda_0)$ if and only if $\ell \leq p$ where $\ell$ is given by Lemma~\ref{lem:min-lem-1}.

First consider the case when $m \not\equiv n \pmod{2}$; then $\ell = \max(0, m-n)$ is either $0$ or odd. Since $p$ is also odd, we conclude that $\ell \leq p$ if and only if $\ell \leq p+1$. Thus $m -n \leq p+1$. When 
$m \equiv n \pmod{2}$, we directly obtain $m-n \leq p+1$. The converse follows likewise. Thus we conclude that $\ell \leq p$ if and only if $m -n \leq p+1$, establishing Proposition~\ref{kkcrys1}.

The proof of Proposition~\ref{kkcrys2} proceeds along similar lines, using Lemma~\ref{lem:min-lem-2} instead of Lemma~\ref{lem:min-lem-1}. The details are left to the interested reader. \qed

\section{Appendix}
\subsection{}
We give a proof of Theorem~\ref{cryiso}. We will recall the action of the root operators $e_i$ and $f_i$ corresponding to the simple roots $\alpha_i$ of a symmetrizable Kac-Moody algebra $\lie g$ on LS paths. We recall that an LS path $\pi$ of shape $\lambda$ is given by a sequence:
\begin{equation}\label{eq:lspath}
    \pi = (\sigma_1 > \sigma_2 > \cdots > \sigma_r \ ; \ 0=a_0 < a_1 <a_2 < \cdots < a_{r-1} < a_r =1).
\end{equation}
where $\sigma_j \in W/W_\lambda$ and $a_j$ are rational numbers for $j=1,2,\ldots,r$. The data defining $\pi$ may be viewed as a piecewise linear map from $[0,1]$ to $\lie h^*$ (the dual of the Cartan subalgebra) as follows:
\begin{equation}\label{eq:lspathdef}
    \pi(t) = \sum_{k=1}^{j-1} (a_k - a_{k-1})\sigma_k(\lambda) + (t - a_{j-1})\sigma_j(\lambda), \ \text{for } \ t \in [a_{j-1},a_j].    
\end{equation}
 Fix a simple root $\alpha_i$ of $\lie g$ and define a function $h(t)$ in the interval $[0,1]$ as follows:
\begin{equation}\label{eq:hfun}
    h(t) = \langle \pi(t), \alpha_i^\vee \rangle
\end{equation}
where $\alpha_i^\vee$ is the simple coroot corresponding to $\alpha_i$.

We will define the action of $f_i$ on $\pi$; the action of $e_i$ is similar and we refer to  \cite{litt:inv} for more details. We remark that we use the definition from \cite{litt:inv} rather than the improved version from \cite{litt:ann}, but these coincide on LS paths. Let $Q$ be the minimum integer attained by the function $h$ and $p$ be the largest index such that $h(a_p) = Q$. Let $P = h(1)-Q$; if $P\geq 1$ choose the minimum $x$ such that $x\geq q$ and $h(t)\geq Q+1$ for $t \geq a_x$.

The LS path $f_i\pi$ may be obtained as follows \cite[Proposition 4.2]{litt:inv}:
\begin{enumerate}
    \item If $h(a_x) = Q+1$ and $s_{\alpha_i}\sigma_{p+1}=\sigma_p$ then, 
    $f_i\pi$ is given by the tuple: 
    \begin{align*} (&\sigma_1>\cdots>\sigma_{p-1}>s_{\alpha_i}\sigma_{p+1} > \cdots > s_{\alpha_i}\sigma_x > \sigma_{x+1} > \cdots > \sigma_r ;\\ &a_0 < a_1 < \cdots <a_{p-1}<a_{p+1}<\cdots<a_r).
    \end{align*}
    \item If $h(a_x) = Q+1$ and $s_{\alpha_i}\sigma_{p+1}<\sigma_p$ then, 
    $f_i\pi$ equals 
    \begin{align*}
    (&\sigma_1>\cdots>\sigma_{p}>s_{\alpha_i}\sigma_{p+1} > \cdots > s_{\alpha_i}\sigma_x > \sigma_{x+1} > \cdots > \sigma_r ;\\ &a_0 < a_1 <\cdots\cdots<a_r).
    \end{align*}
    \item If $h(a_x) > Q+1$ and $s_{\alpha_i}\sigma_{p+1}=\sigma_p$ then $f_i\pi$ equals
    \begin{align*} (&\sigma_1>\cdots>\sigma_{p-1}>s_{\alpha_i}\sigma_{p+1} > \cdots > s_{\alpha_i}\sigma_x > \sigma_{x} > \cdots > \sigma_r ; \\&a_0 < a_1 < \cdots <a_{p-1}<a_{p+1}<\cdots < a_{x-1}<a<a_x<\cdots<a_r)\end{align*} where $h(a)=Q+1.$ 
    \item If $h(a_x) > Q+1$ and $s_{\alpha_i}\sigma_{p+1}<\sigma_p$ then, 
    $f_i\pi$ equals \begin{align*}(&\sigma_1>\cdots>\sigma_{p}>s_{\alpha_i}\sigma_{p+1} > \cdots > s_{\alpha_i}\sigma_x > \sigma_{x} > \cdots > \sigma_r ; \\&a_0 < a_1 < \cdots <a_{x-1}<a<a_{x}<\cdots<a_r)\end{align*} where $h(a)=Q+1$.
\end{enumerate}

We need the following lemma for the proof of \ref{cryiso}.
\begin{lem}\label{strglem}
Let $b = (\mu_1,\mu_2,\ldots,\mu_l) \in \mathcal{Y}(0)$. Let $m=\mu_1$,  $n=l$ and  $\widetilde{b}' = (i_{n+1},i_{n+2},\ldots,i_m)$ be the associated partition defined in Section~ \ref{s:cryiso}. The $0$-signature and $1$-signature of $b$ are given as follows:

\begin{enumerate}
\item\label{it1} If $m$ and $n$ are both even (resp. odd) then the $0$-signature (resp. $1$-signature) of $b$ is given as follows:
\begin{equation}\label{strlm1}
\underbrace{++\cdots+}_{(n+1-i_{n+1})} \underbrace{--\cdots-}_{(i_{n+1}-i_{n+2})} \cdots\cdots \underbrace{++\cdots+}_{(i_{m-2}-i_{m-1})} \underbrace{--\cdots-}_{(i_{m-1}-i_m)}\underbrace{++\cdots+}_{(i_m)} 
\end{equation}

\item\label{it2} If $m$ is even and $n$ is odd (resp. $m$ is odd and $n$ is even) then the $0$-signature (resp. $1$-signature) of $b$ is given as follows:
\begin{equation}\label{strlm2}
\underbrace{--\cdots-}_{(n-i_{n+1})} \underbrace{++\cdots+}_{(i_{n+1}-i_{n+2})} \cdots\cdots \underbrace{++\cdots+}_{(i_{m-2}-i_{m-1})} \underbrace{--\cdots-}_{(i_{m-1}-i_m)}\underbrace{++\cdots+}_{(i_m)} 
\end{equation}

\item\label{it3} If $m$ is odd and $n$ is even (resp. $m$ is even and $n$ is odd) then the $0$-signature (resp. $1$-signature) of $b$ is given as follows:
\begin{equation}\label{strlm3}
\underbrace{++\cdots+}_{(n+1-i_{n+1})} \underbrace{--\cdots-}_{(i_{n+1}-i_{n+2})}\cdots\cdots\underbrace{--\cdots-}_{(i_{m-2}-i_{m-1})} \underbrace{++\cdots+}_{(i_{m-1}-i_m)}\underbrace{--\cdots-}_{(i_m)} 
\end{equation}

\item\label{it4} If $m$ and $n$ are both odd (resp. even) then the $0$-signature (resp. $1$-signature) of $b$ is given as follows:
\begin{equation}\label{strlm4}
\underbrace{--\cdots-}_{(n-i_{n+1})} \underbrace{++\cdots+}_{(i_{n+1}-i_{n+2})}\cdots\cdots\underbrace{--\cdots-}_{(i_{m-2}-i_{m-1})} \underbrace{++\cdots+}_{(i_{m-1}-i_m)}\underbrace{--\cdots-}_{(i_m)}  
\end{equation}
\end{enumerate}
\end{lem} 
 
\begin{proof}
We will give a proof only of part $(1)$ for the $0$-\emph{signature} of $b$. A similar proof works for the $1$-\emph{signature} and for the cases $(2)$, $(3)$ and $(4)$. Observe that $i_{n+j}$  is less than or equal to the length of the $j^{th}$ column of  $b$ for $j=1,2,\ldots,m-n$. If $m$ is even then the last coloured box in the first row is $1$-\emph{coloured}. Then a $0$-\emph{coloured} box is addable in $(m+1)^{th}$-column of $b$, thus we have a $``+"$ at the rightmost of the $0$-signature of $b$. Note that we delete $(n-k+1)$ from the $k^{th}$-row which implies $\mu_{t+1} = \mu_t-1 $ for $t=1,2,\ldots,i_{m}-1$ and hence $``+"$ occurs $i_m$ times by Remark~ \ref{r:regular}. Next we have $\mu_{i_m+t+1} = \mu_{i_m+t}-1$ for $t=0,1,\ldots,(i_{m-1}-i_m-1)$; then $``-"$ occurs $(i_{m-1}-i_m)$ times by Remark~\ref{r:regular}. Continuing this way, we get strings corresponding to all cases depending on the parities of $m$ and $n$.
\end{proof}

\subsection{} Now we will prove Theorem~\ref{cryiso}

\begin{proof}
To prove the map $\Psi$ is a crystal isomorphism we have to show the following:
\begin{enumerate}
\item for $b \in \mathcal{Y}(0)$ if ${f_i} b \not=0$ then $\Psi({f_i} b) = f_i \Psi(b)$ for $i=0,1$.

\item for $b \in \mathcal{Y}(0)$ if ${e_i} b \not=0$ then $\Psi({e_i} b) = e_i \Psi(b)$ for $i=0,1$.
\end{enumerate}

It suffices to prove $(1)$, since $(2)$ follows from $(1)$ from the fact that  ${f_i}$ and ${e_i}$ are inverses of each other. We will prove $(1)$ only for ${f_0}$ and the proof for ${f_1}$ is similar.

Let $b = (\mu_1,\mu_2,\ldots,\mu_l) \in \mathcal{Y}(0)$  be a charged partition. Let $m=\mu_1,n=l$ and $\widetilde{b}' = (i_{n+1},i_{n+2},\ldots,i_m)$ be the associated partition defined in \ref{s:cryiso}. Then:
 \begin{equation}\label{imgphib}
\Psi(b) = \left\lbrace \begin{array}{c}
\bar{w}:= w_{m} > w_{m-1} > \cdots >w_{n+1} > w_{n} \\
\bar{a} : = 0  < \frac{i_m}{m} < \frac{i_{m-1}}{m-1}<\cdots< \frac{i_{n+1}}{n+1}<1. \end{array} \right.
 \end{equation}
 
 To find the action of $f_0$ on the LS path $\Psi(b)$ we have to determine the rightmost minimum of the piecewise linear function $h(t):= \langle \Psi(b)(t),\alpha_0^\vee \rangle$. Note that $h$ is linear in each interval $[\frac{i_{j+1}}{j+1},\frac{i_j}{j}]$ where $m \geq j \geq n$ and $i_{m+1}=0, i_n=n$. Thus the rightmost minimum of $h$ occurs at $\frac{i_{j+1}}{j+1}$ for the largest $j$ at which  $h\left(\frac{i_{j+1}}{j+1}\right)$ attains the minimum value of $h$.
 
  The action ${f_0}$ on $b$ will add a $0$-\emph{coloured} box at the bottom of the column which corresponds to the right most $``+"$ (it will exist because ${f_0} b \not=0$ in assumption) in the reduced $0$-\emph{signature} of $b$. We consider the following three cases. 
  
 \textbf{Case 1.} Suppose the reduced $0$-\emph{signature} of $b$ has at least one $``+"$ in the last block (i.e., in the $i_m^{th}$-block) in case of strings \eqref{strlm1} and \eqref{strlm2} of Lemma \ref{strglem}. Then the reduced $0$-\emph{signature} will be of the form:
\begin{equation}\label{eq:red1}
\underbrace{++\cdots+}\cdots\cdots\underbrace{++\cdots+}\underbrace{++\cdots+}_{i_m-block}
\end{equation}

Note that the reduced $0$-\emph{signature} \eqref{eq:red1} results in the following inequalities:

\begin{equation}\label{eq:inq1}
         i_m - (i_{m-1}-i_m) \geq 1,; \text{ i.e., } 2 i_m - i_{m-1} \geq 1,
    \end{equation}
    
   and
   \begin{equation}\label{eq:inq2}
    2 (i_m + i_{m-2} + i_{m-4}+\cdots+i_{m-2k}) - i_{m-(2k+1)} - 2 ( i_{m-(2k-1)}+\cdots+i_{m-3} +  i_{m-1}) \geq 1,
   \end{equation}  
    where $k=1,\ldots,\frac{m-n-2}{2}$ for string \eqref{strlm1}, and  $k=1,\ldots,\frac{m-n-3}{2}$ for string \eqref{strlm2}.

   The last inequality for equation \eqref{strlm2} is: 
   \begin{equation}\label{eq:inq3}
    2 (i_m + i_{m-2} + i_{m-4}+\cdots+i_{n+3} + i_{n+1}) - n - 2(i_{n+2} + i_{n+4}+\cdots+i_{m-3} +  i_{m-1}) \geq 1.
   \end{equation}

Observe that in this case the action of ${f_0}$ on $b$ will add a $0$-\emph{coloured} box in $(m+1)^{th}$-column hence, ${f_0}b = (\mu_1+1,\mu_2,\ldots,\mu_l)$ and $\widetilde{({f_0}b)}' = (i_{n+1},i_{n+2},\ldots,i_m,i_{m+1})$ where $i_{m+1}=1$. The image of ${f_0}b$ under the map $\Psi$ is thus  given as follows:

\begin{equation}\label{cs1s}
\Psi({f_0}b) = \left\lbrace\begin{array}{c}
 w_{m+1} > w_{m} > w_{m-1} > \cdots >w_{n+1} > w_{n} \\
 0 < \frac{i_{m+1}}{m+1} < \frac{i_m}{m} < \frac{i_{m-1}}{m-1}<\cdots< \frac{i_{n+1}}{n+1}<1. \end{array}\right.
\end{equation}

 Now we will determine $f_0(\Psi(b))$.

 \textbf{Claim 1.} The rightmost minimum of $h$ occurs at $0$.
 
 \textbf{Proof of claim 1:} Clearly $h(0) =0$, and $m$ is even. By equations \eqref{aglev} we have:
 \[ h\left(\frac{i_m}{m}\right)= \frac{i_m}{m}(m+1) > 1 \]
\[ h\left(\frac{i_{m-1}}{m-1}\right)= \frac{i_m}{m}(m+1) - \left( \frac{i_{m-1}}{m-1}- \frac{i_m}{m} \right) (m-1) = 2 i_m - i_{m-1}  \]

Then $h\left(\frac{i_{m-1}}{m-1}\right) \geq 1$ by the inequality \eqref{eq:inq1}.
\[ h\left(\frac{i_{m-2}}{m-2}\right) = h\left(\frac{i_{m-1}}{m-1}\right) +  \left( \frac{i_{m-2}}{m-2}- \frac{i_{m-1}}{m-1} \right)  (m-1) \geq 1 \]
since $\left( \frac{i_{m-2}}{m-2}- \frac{i_{m-1}}{m-1} \right)(m-1) > 0.$

\[ h\left(\frac{i_{m-3}}{m-3}\right) = h\left(\frac{i_{m-2}}{m-2}\right) - \left( \frac{i_{m-3}}{m-3}- \frac{i_{m-2}}{m-2} \right) (m-3) \]

Then $h\left(\frac{i_{m-3}}{m-3}\right) = 2(i_m+i_{m-2})-(i_{m-3}+2i_{m-1}) \geq 1$ by  \eqref{eq:inq2}.
Similarly in general we can see that $ h\left(\frac{i_{m-(2k+1)}}{m-(2k+1)}\right) \geq 1$ by \eqref{eq:inq2} where $k$ varies as given \eqref{eq:inq2}.
Also, we have:
\[ h\left(\frac{i_{m-2k}}{m-2k}\right) = h\left(\frac{i_{m-(2k+1)}}{m-(2k+1)}\right) + \left(\frac{i_{m-2k}}{m-2k} - \frac{i_{m-(2k+1)}}{m-(2k+1)} \right) (m-2k+1) \geq 1\]
since $h\left(\frac{i_{m-(2k+1)}}{m-(2k+1)}\right) \geq 1$ and $\left(\frac{i_{m-2k}}{m-2k)} - \frac{i_{m-(2k+1)}}{m-(2k+1)} \right) (m-2k+1) \geq 0$ where $k$ varies as given in inequality \eqref{eq:inq2}. 

For the string \eqref{strlm2}: $h(\frac{i_{n+1}}{n+1}) = h(\frac{i_{n+2}}{n+2}) + (\frac{i_{n+1}}{n+1} - \frac{i_{n+2}}{n+2})(n+2) \geq 1$.
Hence we can see that $h\left(\frac{i_j}{j}\right) \geq 1$ for all $m \geq j \geq n+1$ for both strings \eqref{strlm1} and \eqref{strlm2}. Now we determine the endpoint of $h$:
\begin{enumerate}
    \item $h(1) = h(\frac{i_{n+1}}{n+1})+(n+1 - i_{n+1})$ for string \eqref{strlm1}, and we have $h(1)\geq 1$.
    \item $h(1) = h\left(\frac{i_{n+2}}{n+2} \right) + 2i_{n+1}-n-i_{n+2} $  for the string \eqref{strlm2} and
by the above inequality \eqref{eq:inq3} we have $h(1) \geq 1$.
\end{enumerate}
So finally we conclude that the rightmost minimum of the function $h$ is attained at $0$ and that $t_0 = \frac{1}{m+1}$ is the leftmost in $[0,1]$ such that $h(t_0)=1$. Then the action of $f_0$ on the path $\Psi(b)$ is:

\begin{equation}\label{cs1f}
f_0(\Psi(b)) = \left\lbrace\begin{array}{c}
 w_{m+1} > w_{m} > w_{m-1} >\cdots >w_{n+1} > w_{n} \\
 0 < \frac{i_{m+1}}{m+1} < \frac{i_m}{m} < \frac{i_{m-1}}{m-1}<\cdots< \frac{i_{n+1}}{n+1}<1 \end{array}\right.
\end{equation}
where $\frac{i_{m+1}}{m+1} = t_0 \implies i_{m+1}=1$. By equation \eqref{cs1s} and \eqref{cs1f} we see that $\Psi({f_0}b)=  f_0(\Psi(b))$.

\smallskip
\textbf{Case 2.} Suppose the reduced $0$-\emph{signature} of $b$ results in $``+"$ signs only in the first block (i.e., the $(n+1 - i_{n+1})$-block) in case of strings \eqref{strlm1} and \eqref{strlm3} of Lemma \ref{strglem}. Then the reduced $0$-\emph{signature} will be of the form:
\begin{equation}\label{eq:red2}
  \underbrace{+\cdots++}_{(n+1-i_{n+1})}\underbrace{--\cdots--}\cdots\cdots\underbrace{--\cdots--}  
\end{equation}
and \begin{equation}\label{eq:inq21}
    n+1 - i_{n+1} \geq 1.
\end{equation}

Note that the reduced $0$-signature \eqref{eq:red2} results in the following inequalities: 
\begin{equation}\label{eq:inq22}
    i_{n+1}-i_{n+2} \geq i_{n+2} - i_{n+3} \Rightarrow i_{n+1}+i_{n+3} - 2 i_{n+2} \geq 0,
\end{equation}
 
  
and the general term is:
\begin{equation}\label{eq:inq23}
     i_{n+1} + 2(i_{n+2k-1}+\cdots+ i_{n+3}) + i_{n+2k+1} \geq 2 (i_{n+2}+i_{n+4}+\cdots+i_{n+2k})
\end{equation}
where $k=2,3,\ldots,\frac{m-n-2}{2}$ for string \eqref{strlm1} and  $k=2,3,\ldots,\frac{m-n-1}{2}$ for string \eqref{strlm3}.

The last inequality for \eqref{strlm1} is:
\begin{equation}\label{eq:inq24}
     2(i_{m-1}+\cdots+ i_{n+3}) + i_{n+1} \geq 2 (i_{n+2}+i_{n+4}+\cdots+i_{m}).
\end{equation}
Since the rightmost $``+"$ in the reduced $0$-signature of $b$ occurs in the $(n+1 - i_{n+1})$-column, the action of ${f_0}$ on $b$ will add a $0$-\emph{coloured} box at the bottom of the $(n+1-i_{n+1})^{th}$-column. We have the following two cases: \begin{enumerate}
\item  Let $n > i_{n+1}$, then ${f_0}b = (\mu_1,\mu_2,\ldots,\mu_{(i_{(n+1)}+1)}+1,\ldots,\mu_l)$ and \[ \widetilde{({f_0}b)}' = (i_{n+1}+1,i_{n+2},\ldots,i_m)\] 
The image of ${f_0}b$ under $\Psi$ is:
\begin{equation}\label{cs2s1}
\Psi({f_0}b) = \left\lbrace\begin{array}{c}
 w_{m} > w_{m-1} >\cdots>w_{n+1} > w_{n} \\
 0 < \frac{i_m}{m} < \frac{i_{m-1}}{m-1}<\cdots< \frac{i_{n+1}+1}{n+1}<1. \end{array}\right.
 \end{equation}
 \item  Let $n = i_{n+1}$, then ${f_0}b = (\mu_1,\mu_2,\ldots,\mu_l,\mu_{l+1})$ where $\mu_{l+1}=1$ and \[ \widetilde{({f_0}b)}' = (i_{n+2},i_{n+3},\ldots,i_m)\] 
\begin{equation}\label{cs2s2}
\Psi({f_0}b) = \left\lbrace\begin{array}{c}
  w_{m} > w_{m-1} >\cdots>w_{n+1} \\
 0 < \frac{i_m}{m} < \frac{i_{m-1}}{m-1}<\cdots< \frac{i_{n+2}}{n+2}<1. \end{array}\right.
 \end{equation}
\end{enumerate}

Now we will determine $f_0\Psi(b)$; as before for this we have to find the point at which the rightmost minimum of $h$ occurs.

\textbf{Claim 2.} The rightmost minimum of $h$ is attained at $\frac{i_{n+1}}{n+1}$.

\textbf{Proof of claim 2:} Note that by inequality \eqref{eq:inq21}
\[ h(1) - h\left(\frac{i_{n+1}}{n+1}\right) = \left(1-\frac{i_{n+1}}{n+1}\right)(n+1)=n+1-i_{n+1} \geq 1 \]
\[ \Rightarrow h(1) \geq 1+h\left(\frac{i_{n+1}}{n+1}\right) \]

Now consider,
\begin{equation}\label{cs2h1}
h\left(\frac{i_{n+2k+1}}{n+2k+1}\right) - h\left(\frac{i_{n+2(k+1)}}{n+2(k+1)}\right) = -\left(\frac{i_{n+2k+1}}{n+2k+1} - \frac{i_{n+2(k+1)}}{n+2(k+1)} \right)(n+2k+1) \leq 0
\end{equation}
\[ \implies h\left(\frac{i_{n+2k+1}}{n+2k+1}\right) \leq h\left(\frac{i_{n+2(k+1)}}{n+2(k+1)}\right) \]
where $k = 0,1,2,\ldots,\frac{m-n-2}{2}$ for string \eqref{strlm1} and $k=0,1,2,\ldots,\frac{m-n-3}{2}$ for string \eqref{strlm3}.

Next, we consider,
\begin{equation}\label{cs2h2}
h\left(\frac{i_{n+2k}}{n+2k}\right) - h\left(\frac{i_{n+2k+1}}{n+2k+1}\right)  = \left(\frac{i_{n+2k}}{n+2k} - \frac{i_{n+2k+1}}{n+2k+1} \right)(n+2k+1) \geq 0
\end{equation}
where $k=1,2,\ldots,\frac{m-n-2}{2}$ for string \eqref{strlm1} and $k=1,2,\ldots,\frac{m-n-1}{2}$ for string \eqref{strlm3}.

From equations \ref{cs2h1} and \ref{cs2h2} we have:
\[ h\left(\frac{i_{n+1}}{n+1}\right) - h\left(\frac{i_{n+2k+1}}{n+2k+1}\right) \] \[ =  2 (i_{n+2}+i_{n+4}+\cdots+i_{n+2k}) - i_{n+1} - 2(i_{n+2k-1}+\cdots+ i_{n+3}) - i_{n+2k+1}  \]
and from the above inequality \eqref{eq:inq23} we have 
\begin{equation}\label{eq:cs2h3}
    h\left(\frac{i_{n+1}}{n+1}\right) \leq h\left(\frac{i_{n+2k+1}}{n+2k+1}\right)
\end{equation}
where $k=1,2,3,\ldots,\frac{m-n-2}{2}$ for string \eqref{strlm1} and  $k=1,2,3,\ldots,\frac{m-n-1}{2}$ for string \eqref{strlm3}.


Now, from  \eqref{cs2h2} and \eqref{eq:cs2h3} we have
\begin{equation}\label{eq:cs2h4}
     h\left(\frac{i_{n+2k}}{n+2k}\right) \geq h\left(\frac{i_{n+2k+1}}{n+2k+1}\right) \geq h\left(\frac{i_{n+1}}{n+1}\right)
\end{equation}
where $k = 1,2,\ldots,\frac{m-n-2}{2}$ for string \eqref{strlm1} and $k=1,2,\ldots,\frac{m-n-1}{2}$ for string \eqref{strlm3}.

For string \eqref{strlm1} 
\[ h\left(\frac{i_{m-1}}{m-1}\right) - h\left(\frac{i_{m}}{m}\right)  = - \left(\frac{i_{m-1}}{m-1} - \frac{i_{m}}{m} \right)(m-1) \leq 0 \]

\begin{equation}\label{eq:cs2h5}
    \Longrightarrow h\left(\frac{i_{m}}{m}\right) \geq h\left(\frac{i_{m-1}}{m-1}\right) \geq h\left(\frac{i_{n+1}}{n+1}\right) 
\end{equation}

and lastly,
\[ h\left(\frac{i_{n+1}}{n+1}\right) = 2 (i_{n+2}+i_{n+4}+\cdots+i_{m}) - 2(i_{m-1}+\cdots+ i_{n+3}) - i_{n+1} \]

Then, by \eqref{eq:inq24} we have $0=h\left(0\right) \geq h\left(\frac{i_{n+1}}{n+1}\right)$. Likewise, for string \eqref{strlm3}:
\[ h\left(\frac{i_m}{m}\right) =  - \left(\frac{i_m}{m}\right)m = - i_m \leq  0 = h(0) \]
\[ \Rightarrow h(0) \geq h\left(\frac{i_m}{m}\right) \geq \left(\frac{i_{n+1}}{n+1}\right) \]

Finally we have $h\left( \frac{i_{n+1}}{n+1} \right) \leq h(0), \, h\left( \frac{i_{n+1}}{n+1} \right) \leq h(1) $ and $ h\left( \frac{i_{n+1}}{n+1} \right) \leq h\left(\frac{i_{j}}{j}\right)$ for all $m \geq j \geq n+2.$ Hence the rightmost minimum of $h$ occurs at $\frac{i_{n+1}}{n+1}$. Thus the action of $f_0$ on path $\Psi(b)$ becomes:
\begin{enumerate}
\item When $n > i_{n+1}$, we have \[ h\left(\frac{i_{n+1}+1}{n+1}\right) - h\left(\frac{i_{n+1}}{n+1}\right) = 1 \]
and hence,
\begin{equation}\label{cs2f1}
f_0(\Psi(b)) = \left\lbrace\begin{array}{c}
 w_{m} > w_{m-1} >\cdots >w_{n+1} > w_{n} \\
 0 < \frac{i_m}{m} < \frac{i_{m-1}}{m-1}<\cdots< \frac{i_{n+1}+1}{n+1}<1. \end{array}\right.
\end{equation}
By equations \eqref{cs2s1} and \eqref{cs2f1} we have $\Psi({f_0}b)=  f_0(\Psi(b))$.
\item  When $n = i_{n+1}$, we have
\[ h\left( 1 \right) - h\left(\frac{i_{n+1}}{n+1}\right) = 1 \]
and hence,
\begin{equation}\label{cs2f2}
f_0(\Psi(b)) = \left\lbrace\begin{array}{c}
 w_{m} > w_{m-1} >\cdots>w_{n+1}  \\
 0 < \frac{i_m}{m} < \frac{i_{m-1}}{m-1}<\cdots< \frac{i_{n+2}}{n+2}<1. \end{array}\right.
\end{equation}
By equation \eqref{cs2s2} and \eqref{cs2f2} we see that $\Psi({f_0}b)=  f_0(\Psi(b))$
\end{enumerate}

\textbf{Case 3.} Suppose the rightmost plus signs in the reduced $0$-\emph{signature} of $b$ occur in the $(i_{j-1}-i_j)$ block for strings \eqref{strlm1}, \eqref{strlm2}, \eqref{strlm3} and \eqref{strlm4}. Note that $(j-1)$ is even, $i_{j-1}>i_j$ and the reduced $0$-\emph{signature} will be of the form:
\begin{equation}\label{eq:red3}
    \underbrace{++\cdots+}\cdots\underbrace{++\cdots+} \underbrace{++\cdots+}_{(i_{j-1}-i_{j})-block}\underbrace{--\cdots-}\cdots\cdots\underbrace{--\cdots-} 
\end{equation}

The reduced $0$-\emph{signature} in \eqref{eq:red3} implies the following inequalities: (i) the right side inequality of the $(i_{j-1} - i_j)$-block:
\begin{equation}\label{eq:inq31}
        i_j - i_{j+1} \geq i_{j+1} - i_{j+2}\Rightarrow i_j + i_{j+2}\geq 2i_{j+1},
    \end{equation}
and (ii) the general inequality:
\begin{equation}\label{eq:inq32}
     i_{j+2k} + 2(i_{j+2(k-1)}+\cdots+i_{j+2})+i_j \geq 2(i_{j+1}+i_{j+3}+\cdots+i_{j+2k-1}),\end{equation}
where $k=2,3,\ldots,\frac{m-j-1}{2}$ for strings \eqref{strlm1}, \eqref{strlm2} and  $k=2,3,\ldots,\frac{m-j}{2}$ for strings \eqref{strlm3}, \eqref{strlm4}.

The last inequality for \eqref{strlm1}, \eqref{strlm2} is:
\begin{equation}\label{eq:inq33}
    2(i_{m-1} +i_{m-3}+...+i_{j+2})+i_j \geq 2(i_{j+1}+i_{j+3}+....+i_{m-2}+i_m).
\end{equation}

Now, the left side inequality of the $(i_{j-1}-i_j)$-block is,
\begin{equation}\label{eq:inq34}
   1+i_{j-2}-i_{j-1}\leq i_{j-1}-i_j \Rightarrow 1+ i_{j-2}+i_j \leq 2 i_{j-1},
\end{equation}
and the general term is:
\begin{equation}\label{eq:inq35}
    1+ i_{j-2k}+2(i_{j-2(k-1)}+\cdots+i_{j-2})+i_j \leq 2 (i_{j-1}+i_{j-3}+\cdots+i_{j-2k+1}) 
\end{equation}
where $k=2,3,\ldots,\frac{j-n-1}{2}$ for strings \eqref{strlm1},\eqref{strlm3} and $k=2,3,\ldots,\frac{j-n-2}{2}$ for strings \eqref{strlm2}, \eqref{strlm4}.

The last inequality for strings \eqref{strlm2}, \eqref{strlm4} is:
\begin{equation}\label{eq:inq36}
    1+n+2(i_{n+2}+i_{n+4}+\cdots+i_{j-2})+i_j \leq 2 (i_{j-1}+i_{j-3}+\ldots+i_{n+1}).
\end{equation}

Now, we derive ${f_0}b$.
The action of ${f_0}$ on $b$ will add a $0$-\emph{coloured} box at the bottom of $(\mu_{i_j+1})^{th}$-column which has the right most $``+"$. Hence  ${f_0}b = (\mu_1,\mu_2,\ldots,\mu_{(i_j+1)}+1,\ldots,\mu_l)$ and \[ \widetilde{({f_0}b)}' = (i_{n+1},\ldots,i_{j-1},i_j+1,i_{j+1},\ldots,i_m)\] 

Then the image of ${f_0}b$ under the map $\Psi$ is:
\begin{equation}\label{cs3s1}
\Psi({f_0}b) = \left\lbrace\begin{array}{c}
 w_{m} >\cdots> w_{j-1} > w_j>w_{j+1} >\cdots> w_{n} \\
 0 < \frac{i_m}{m} <\cdots<\frac{i_{j-1}}{j-1}< \frac{i_j+1}{j}<\cdots<\frac{i_{n+1}}{n+1}<1. \end{array}\right.
 \end{equation}

 Now, we will derive $f_0\Psi(b)$.  We have to find the point at which the rightmost minimum $h$ occurs.

\textbf{Claim 3.} The rightmost minimum of $h$ occurs at $\frac{i_{j}}{j}$.

\textbf{Proof of claim 3.} Firstly we prove for the interval $\left[0,\frac{i_j}{j}\right]$. We know that $j$ is odd.

Consider,
\begin{equation}\label{eq:cs3cl1}
    h\left(\frac{i_{j+2k}}{j+2k}  \right)- h\left( \frac{i_{j+2k+1}}{j+2k+1}\right) =  - \left( \frac{i_{j+2k}}{j+2k} - \frac{i_{j+2k+1}}{j+2k+1} \right)(j+2k) \leq 0
\end{equation}
where $k=0,1,2,\ldots,\frac{m-j-1}{2}$ for strings \eqref{strlm1}, \eqref{strlm2} and $k=0,1,2,\ldots,\frac{m-j-2}{2}$ for strings \eqref{strlm3}, \eqref{strlm4}.
Also, we have
\begin{equation}\label{eq:cs3cl2}
   h\left(\frac{i_{j+2k+1}}{j+2k+1}  \right)- h\left( \frac{i_{j+2(k+1)}}{j+2(k+1)}\right) =  \left( \frac{i_{j+2k+1}}{j+2k+1} - \frac{i_{j+2(k+1)}}{j+2(k+1)} \right)(j+2(k+1))
\end{equation}
where $k=0,1,2,\ldots,\frac{m-j-3}{2}$ for strings \eqref{strlm1}, \eqref{strlm2} and $k=0,1,2,\ldots,\frac{m-j-2}{2}$ for strings \eqref{strlm3}, \eqref{strlm4}.
From  \eqref{eq:cs3cl1} and \eqref{eq:cs3cl2} for $k=0$, we have:
\[ h\left(\frac{i_j}{j}\right)-h\left(\frac{i_{j+2}}{j+2} \right) = 2i_{j+1} - i_j -i_{j+2},\]
and by inequality \ref{eq:inq31} we have $ h\left(\frac{i_j}{j}\right)\leq h\left(\frac{i_{j+2}}{j+2} \right)$.

Again from \eqref{eq:cs3cl1} and \eqref{eq:cs3cl2} we have:
\[\begin{split}
    h\left(\frac{i_j}{j}\right)-h\left(\frac{i_{j+2(k+1)}}{j+2(k+1)} \right) = 2(i_{j+1}+i_{j+3}+\cdots+i_{j+2k+1}) \\ - i_{j+2k} - 2(i_{j+2(k+1)}+\cdots+i_{j+2})-i_j \leq 0, 
    \end{split}\]
where $k=1,2,\ldots,\frac{m-j-3}{2}$ for strings \eqref{strlm1}, \eqref{strlm2} and  $k=1,2,\ldots,\frac{m-j-2}{2}$ for strings \eqref{strlm3}, \eqref{strlm4}.
Hence, \begin{equation}\label{eq:cs3cl3}
     h\left(\frac{i_j}{j}\right)\leq h\left(\frac{i_{j+2(k+1)}}{j+2(k+1)} \right)
\end{equation} for $k=1,2,\ldots,\frac{m-j-3}{2}$ for strings \eqref{strlm1}, \eqref{strlm2} and  $k=1,2,\ldots,\frac{m-j-2}{2}$ for strings \eqref{strlm3}, \eqref{strlm4}.

From  \eqref{eq:cs3cl1} and \eqref{eq:cs3cl3} we have:
\begin{equation}\label{eq:cs3cl4}
 h\left(\frac{i_j}{j}\right)\leq h\left(\frac{i_{j+2k}}{j+2k}  \right) \leq h\left( \frac{i_{j+2k+1}}{j+2k+1}\right)
\end{equation}
where  for $k=1,2,\ldots,\frac{m-j-1}{2}$ for strings \eqref{strlm1}, \eqref{strlm2}, \eqref{strlm3} and \eqref{strlm4}.
%
For strings \eqref{strlm3} and \eqref{strlm4} we have, $h\left(\frac{i_m}{m}\right) = -\frac{i_m}{m}\cdot m = -i_m \leq 0 = h(0)$. Hence  inequality \eqref{eq:cs3cl4} implies: 
\[ h\left(\frac{i_j}{j}\right)\leq h\left(\frac{i_m}{m}\right)\leq h(0)\]



For strings \eqref{strlm1} and \eqref{strlm2} we have:
\[ h\left(\frac{i_j}{j}\right) =  2(i_{j+1}+i_{j+3}+....+i_{m-2}+i_m) - 2(i_{m-1} +i_{m-3}+...+i_{j+2})-i_j \]
and by \eqref{eq:inq33} we have $h\left(\frac{i_j}{j}\right) \leq 0=h(0).$
We see that $h\left( \frac{i_j}{j} \right)$ is the rightmost minimum of $h$ in the interval $\left[0,\frac{i_j}{j}\right]$.
Next, we will prove that $h(\frac{i_k}{k}) \geq h\left( \frac{i_j}{j} \right)+1$ for $j-1 \geq k \geq n$ where $i_n = n$.

Consider,
\begin{equation}\label{eq:cs3cl5}
    h\left(\frac{i_{j-2k}}{j-2k}  \right)- h\left( \frac{i_{j-2k+1}}{j-2k+1}\right) =  - \left( \frac{i_{j-2k}}{j-2k} - \frac{i_{j-2k+1}}{j-2k+1} \right)(j-2k) \leq 0
\end{equation}
where $k=1,2,\ldots,\frac{j-n-1}{2}$ for strings \eqref{strlm1}, \eqref{strlm3} and $k=1,2,\ldots,\frac{j-n-2}{2}$ for strings \eqref{strlm2}, \eqref{strlm4}.
Also, we have
\begin{equation}\label{eq:cs3cl6}
   h\left(\frac{i_{j-2k-1}}{j-2k-1}  \right)- h\left( \frac{i_{j-2k}}{j-2k}\right) =  \left( \frac{i_{j-2k-1}}{j-2k-1} - \frac{i_{j-2k}}{j-2k} \right)(j-2k)
\end{equation}
where $k=0,1,2,\ldots,\frac{j-n-3}{2}$ for strings \eqref{strlm1}, \eqref{strlm3} and $k=0,1,2,\ldots,\frac{j-n-2}{2}$ for strings \eqref{strlm2}, \eqref{strlm4}.
Note that, from equation \eqref{eq:cs3cl6} for $k=0$ we have:
\[h\left(\frac{i_{j-1}}{j-1} \right) - h\left(\frac{i_j}{j}\right) =  \left(\frac{i_{j-1}}{j-1} - \frac{i_j}{j} \right)\cdot j = (i_{j-1} - i_j)+\frac{i_{j-1}}{j-1} \geq 1 \]
\[\implies h\left(\frac{i_{j-1}}{j-1} \right) \geq 1+ h\left(\frac{i_j}{j}\right) \]

From equations \eqref{eq:cs3cl5} for $k=1$ and \eqref{eq:cs3cl6} for $k=0$ we have:

\[h\left(\frac{i_{j-2}}{j-2} \right) - h\left(\frac{i_j}{j}\right) =  2i_{j-1} - i_j - i_{j-2} \]

By the inequality \eqref{eq:inq34} above, we have $ h\left(\frac{i_{j-2}}{j-2} \right) \geq h\left(\frac{i_j}{j}\right)+1$.
Again from equations \eqref{eq:cs3cl5} for $k=2,3,\ldots,\frac{j-n-1}{2}$ and \eqref{eq:cs3cl6} for $k=1,2,\ldots,\frac{j-n-2}{2}$ we have:
\[h\left(\frac{i_{j-2k}}{j-2k} \right) - h\left(\frac{i_j}{j}\right) =  \] \[   2 (i_{j-1}+i_{j-3}+\cdots+i_{j-2k+1}) -  i_{j-2k}-2(i_{j-2(k-1)}+\cdots+i_{j-2})-i_j \]
for $k=2,3,\ldots,\frac{j-n-1}{2}$ for strings \eqref{strlm1}, \eqref{strlm3} and $k=2,3,\ldots,\frac{j-n-2}{2}$ for strings \eqref{strlm2}, \eqref{strlm4}. By \eqref{eq:inq35} we have $h\left(\frac{i_{j-2k}}{j-2k} \right) \geq h\left(\frac{i_j}{j}\right)+1$.
From equation \eqref{eq:cs3cl6} we have:
\[ h\left(\frac{i_{j-2k-1}}{j-2k-1} \right) \geq h\left(\frac{i_{j-2k}}{j-2k} \right) \geq h\left(\frac{i_j}{j}\right)+1 \]
for $k=1,2,\ldots,\frac{j-n-3}{2}$ for strings \eqref{strlm1},\eqref{strlm3} and $k=1,2,\ldots,\frac{j-n-2}{2}$ for strings \eqref{strlm2}, \eqref{strlm4}.


For strings \eqref{strlm1} and \eqref{strlm3},
\[ h(1) - h\left( \frac{i_{n+1}}{n+1} \right) = \left(1- \frac{i_{n+1}}{n+1} \right) (n+1)=n+1-i_{n+1}\geq 0.\]
\[\Rightarrow h(1) \geq h\left( \frac{i_{n+1}}{n+1} \right) \geq h\left(\frac{i_j}{j}\right)+1. \]

For strings \eqref{strlm2} and \eqref{strlm4},
\[ h(1) - h\left( \frac{i_j}{j} \right) = n+2(i_{n+2}+i_{n+4}+...+i_{j-2})+i_j - 2 (i_{j-1}+i_{j-3}+...+i_{n+1}) \]
and by inequality \eqref{eq:inq36} we have $h(1) \geq h\left( \frac{i_j}{j} \right)+1$.

Finally we have proved that $h(\frac{i_k}{k}) \geq h\left( \frac{i_j}{j} \right)+1$ for $j-1 \geq k \geq n$ where $i_n = n$. Thus $h$ has its rightmost minimum at $\frac{i_j}{j}$ and $h(t) \geq h\left( \frac{i_j}{j} \right)+1$ for $t\in \left[\frac{i_{j-1}}{j-1},1\right]$. Also,  $t_1 =\frac{i_j+1}{j} \in \left[\frac{i_j}{j},1\right]$ is the leftmost such that $h(t_1) - h\left( \frac{i_j}{j} \right) = 1. $

Then the action of $f_0$ on path $\Psi(b)$ is:

\begin{equation}\label{cs3f1}
f_0(\Psi(b)) = \left\lbrace\begin{array}{c}
 w_{m} >\cdots> w_{j-1} > w_j>w_{j+1} >\cdots> w_{n} \\
 0 < \frac{i_m}{m} <\cdots<\frac{i_{j-1}}{j-1}< \frac{i_j+1}{j}<\cdots<\frac{i_{n+1}}{n+1}<1. \end{array}\right.
\end{equation}

 From the equations \eqref{cs3s1} and \eqref{cs3f1} we see that $\Psi({f_0}b)=  f_0(\Psi(b))$, which completes the proof.
\end{proof}

\bibliographystyle{plain}

\begin{thebibliography}{10}
\expandafter\ifx\csname url\endcsname\relax
  \def\url#1{{\tt #1}}\fi
\expandafter\ifx\csname urlprefix\endcsname\relax\def\urlprefix{URL }\fi

\bibitem{jimbo-et-al} 
M.~Jimbo, K.C.~Misra, T.~Miwa, M.~Okado, {\em Combinatorics of representations of $U_q(\widehat{\mathfrak{sl}}(n))$ at $q=0$}, Commun. Math. Phys. {\bf 136}, 1991, pp.~543–566

\bibitem{kac}
V.~G. Kac, {\em Infinite-dimensional {L}ie algebras}, 3rd edn., Cambridge
  University Press, Cambridge, 1990,
  \urlprefix\url{http://dx.doi.org/10.1017/CBO9780511626234}.

\bibitem{litt:inv}
P.~Littelmann, {\em A {L}ittlewood-{R}ichardson rule for symmetrizable
  {K}ac-{M}oody algebras}, Invent. Math., {\bf 116}, no. 1-3, 1994,
  pp.~329--346,  \urlprefix\url{https://doi.org/10.1007/BF01231564}.

\bibitem{litt:ann}
P.~Littelmann, {\em Paths and root operators in representation theory}, Ann.
  of Math. (2), {\bf 142}, no.~3, 1995, pp.~499--525,
  \urlprefix\url{https://doi.org/10.2307/2118553}.
  
\bibitem{joseph}
A.~Joseph, {\em Modules with a {D}emazure flag}, in: {\em Studies in {L}ie
  theory\/}, vol. 243 of Progr. Math., Birkh\"{a}user Boston, Boston, MA, 2006,
  pp. 131--169, \urlprefix\url{https://doi.org/10.1007/0-8176-4478-4_8}.

\bibitem{llm}
V.~Lakshmibai, P.~Littelmann, and P.~Magyar, {\em Standard monomial theory for
  {B}ott-{S}amelson varieties}, Compositio Math., {\bf 130}, no.~3, 2002,
  pp.~293--318,  \urlprefix\url{https://doi.org/10.1023/A:1014396129323}.

\bibitem{mw}
K.~C. Misra, E.~A. Wilson, {\em On tensor product decomposition of $\widehat{\mathfrak{sl}}(n)$ modules}, J. Algebra Appl., {\bf 12},  no. 8, 2013, 1350054, 16 pp.

\bibitem{ys}
Y.~B. Sanderson, {\em Dimensions of Demazure modules for rank
two affine Lie algebras}, Compositio. Mathematica., {\bf 101},  1996,
  pp.~115--131,  \urlprefix\url{http://www.numdam.org/item?id=CM_1996__101_2_115_0}.

\bibitem{krv}
M.~S. Kushwaha, K.~N. Raghavan, S.~Viswanath {\em A study of {K}ostant-{K}umar modules via {L}ittelmann paths}, {Advances in Mathematics}, {381}, (2021), {107614}, pp. {0001-8708},
   \urlprefix\url{https://doi.org/10.1016/j.aim.2021.107614}.
     
\bibitem{mm}
K.C.~Misra, T.~Miwa {\em Crystal base for basic representation of $U_q(\widehat{\mathfrak{sl}}(n))$}, Commun. Math. Phys., {\bf 134},  (1990),
  pp.~79--88.


  
\end{thebibliography}

\end{document}